\documentclass[12pt,reqno]{amsart} 
\usepackage[T1]{fontenc}
\usepackage{amsfonts}
\usepackage{amssymb}
\usepackage{mathrsfs}
\usepackage[latin1]{inputenc}
\usepackage{amsmath}
\usepackage{paralist}
\usepackage[english]{babel}
\usepackage{setspace}

\newtheorem{theorem}{Theorem}[section]
\newtheorem{lemma}[theorem]{Lemma}

\newtheorem{remark}[theorem]{Remark}
\newtheorem{prop}[theorem]{Proposition}

\newtheorem{corollary}[theorem]{Corollary}

\numberwithin{equation}{section}
\newcommand{\R}{{\mathbb R}}

\newcommand{\C}{{\mathbb C}}
\newcommand{\N}{{\mathbb N}}

\newcommand{\cL}{{\mathcal L}}

\newcommand{\cE}{{\mathcal E}}

\newcommand{\cB}{{\mathcal B}}

\newcommand{\cS}{{\mathcal S}}
\newcommand{\cU}{{\mathcal U}}
\newcommand{\cV}{{\mathcal V}}

\newcommand{\be}{\beta}

\newcommand{\si}{\sigma}
\newcommand{\su}{\subseteq}

\newcommand{\ov}{\overline}

\newcommand\ind{\mathop{\rm ind\,}}
\newcommand\Ker{\mathop{\rm Ker}}

\newcommand{à}{\`a}

\begin{document}

\title{Spectral properties of generalized Ces\`aro operators in  sequence spaces}

\author{Angela\,A. Albanese, Jos\'e Bonet* and Werner\,J. Ricker}


\address{ Angela A. Albanese\\
Dipartimento di Matematica ``E.De Giorgi''\\
Universit\`a del Salento- C.P.193\\
I-73100 Lecce, Italy}
\email{angela.albanese@unisalento.it}

\address{Jos\'e Bonet \\
Instituto Universitario de Matem\'{a}tica Pura y Aplicada
IUMPA \\
Edificio IDI5 (8E), Cubo F, Cuarta Planta \\
Universidad Politécnica de Valencia \\
E-46071 Valencia, Spain} \email{jbonet@mat.upv.es}

\address{Werner J.  Ricker \\
Math.-Geogr. Fakultät \\
 Katholische Universität
Eichst\"att-Ingol\-stadt \\
D-85072 Eichst\"att, Germany}
\email{werner.ricker@ku.de}

\begin{abstract}
The generalized Cesàro operators $C_t$, for $t\in [0,1]$, were first investigated in the 1980's. They act continuously in many classical Banach sequence spaces contained in $\C^{\N_0}$, such as $\ell^p$, $c_0$, $c$, $bv_0$, $bv$ and, as recently shown, \cite{CR4}, also in the discrete Cesàro spaces $ces(p)$ and their (isomorphic) dual spaces $d_p$. In most cases $C_t$ ($t\not=1$) is compact and its spectra and point spectrum, together with the corresponding eigenspaces, are known. We study these properties of $C_t$, as well as their linear dynamics and mean ergodicity, when they act in certain non-normable sequence spaces contained in $\C^{\N_0}$. Besides $\C^{\N_0}$ itself, the Fr\'echet spaces considered are $\ell(p+)$, $ces(p+)$ and $d(p+)$, for $1\leq p<\infty$, as well as the (LB)-spaces $\ell(p-)$, $ces(p-)$ and $d(p-)$, for $1<p\leq\infty$.
\end{abstract}
\maketitle

\section{Introduction}

The (discrete) generalized Cesàro operators $C_t$, for $t\in [0,1]$, were first investigated by Rhaly, \cite{R1}. The action of $C_t$  from $\omega:=\C^{\N_0}$ into itself (with $\N_0:=\{0,1,2,\ldots\}$) is given by
\begin{equation}\label{Ces-op}
	C_tx:=\left(\frac{t^nx_0+t^{n-1}x_1+\ldots +x_n}{n+1}\right)_{n\in\N_0},\quad x=(x_n)_{n\in\N_0}\in\omega.
	\end{equation}
For $t=0$ note that $C_0$ is the diagonal operator
\begin{equation}\label{Dia-op}
	D_\varphi x:= \left(\frac{x_n}{n+1}\right)_{n\in\N_0}, \quad x=(x_n)_{n\in\N_0}\in\omega,
	\end{equation}
where $\varphi:=\left(\frac{1}{n+1}\right)_{n\in\N_0}$, and for $t=1$ that $C_1$ is the classical Cesàro averaging operator
\begin{equation}\label{Ces-1}
	C_1x:=\left(\frac{x_0+x_1+\ldots+x_n}{n+1}\right),\quad x=(x_n)_{n\in\N_0}\in\omega.
\end{equation}

The spectra of $C_1$ have been investigated in various Banach sequence spaces. For instance, we mention $\ell^p$ ($1<p<\infty$), \cite{BHS,CR1,G,Le}, $c_0$ \cite{Ak,Le,Re}, $c$ \cite{Le}, $\ell^\infty$ \cite{P,Re}, the Bachelis spaces $N^p$ ($1<p<\infty$) \cite{CR2}, $bv$ and $bv_0$ \cite{O1,O2}, weighted $\ell^p$ spaces \cite{ABR1,ABR2}, the discrete Cesàro spaces $ces(p)$ (for $p\in\{0\}\cup (1,\infty)$), \cite{CR3}, and their dual spaces $d_s$ ($1<s<\infty$), \cite{BR1}. For the class of generalized Cesàro operators $C_t$, for $t\in (0,1)$, a study of their spectra and compactness properties (in $\ell^2$) go back to Rhaly, \cite{R1,R2}. A similar investigation occurs for $\ell^p$ ($1<p<\infty$) in \cite{YD} and for $c$ and $c_0$ in \cite{SEl-S,YM}. The paper \cite{SEl-S} also treats $C_t$ when it acts on $bv_0$, $bv$, $c$, $\ell^1$, $\ell^\infty$ and the Hahn sequence space $h$. In the recent paper \cite{CR4} the setting for considering the operators $C_t$ is a large class of Banach lattices in $\omega$, which includes all rearrangement invariant sequence spaces (over $\N_0$ for counting measure), and many others.

Our aim is to study the compactness, the spectra and the dynamics of the generalized Cesàro operators $C_t$, for $t\in [0,1)$, when they act in certain classical, \textit{non-normable} sequence spaces $X\subseteq \omega$. Besides $\omega$ itself, the Fr\'echet spaces considered are $\ell(p+)$, $ces(p+)$ and $d(p+)$, for $1\leq p<\infty$, as well as the (LB)-spaces $\ell(p-)$, $ces(p-)$ and $d(p-)$, for $1<p\leq \infty$.

In Section 2 we formulate various preliminaries that will be needed in the sequel concerning particular properties of the spaces $X$ that we consider, as well as linear operators between such spaces. We also collect some general results required to determine the spectra of operators $T$ acting in the spaces $X$ and the compactness of their dual operator $T'$ acting in the strong dual space $X'_\beta$ of $X$.

Section 3 is devoted to a detailed study of the operators $C_t$, for $t\in [0,1)$, when they act in $\omega$. These operators are \textit{never} compact (c.f. Proposition \ref{nocompact-omega}) and their spectrum is completely described in Theorem \ref{Sp-omega} where, in particular, it is established that the set of all eigenvalues of $C_t$ is independent of $t$ and equals $\Lambda:=\{\frac{1}{n+1}\,:\, n\in\N_0\}$. The $1$-dimensional eigenspace corresponding to $\frac{1}{n+1}$, for each $n\in\N_0$, is identified in Lemma \ref{L-3-3}.

The situation for the other mentioned spaces $X\subseteq \omega$, which is rather different, is treated in Sections 4 and 5. The operator $C_t$, for $t\in [0,1)$, is \textit{always} compact in these spaces; see Theorem \ref{Sp-F}(i) for the case of Fr\'echet spaces and Theorem \ref{Spectrum-LB}(i) for the case of (LB)-spaces. The spectra of $C_t$ are fully determined in Theorems \ref{Sp-F}(ii) and \ref{Spectrum-LB}(ii), and the $1$-dimensional eigenspace corresponding to each eigenvalue of $C_t$ is identified in Theorems \ref{Sp-F}(iii) and \ref{Spectrum-LB}(iii). We note, for all cases of $X$ and $t\in [0,1)$, that the set of all eigenvalues of $C_t$ is again $\Lambda$. The main tool is a factorization result stating  that $C_t=D_\varphi R_t$, where $D_\varphi\colon X\to X$ is a compact (diagonal) operator in $X$ and $R_t\colon X\to X$ is a continuous linear operator; see Propositions \ref{P-4-4}(iii) and \ref{P-5-2}(iii).

For the definition of a mean ergodic operator and the notion of a supercyclic operator we refer to Section 6, where the relevant operators under consideration are $C_t$ acting in the spaces $X$, for each $t\in [0,1)$. It is necessary to determine some abstract results for linear operators in general lcHs' (c.f., Theorems \ref{BoundedOp} and \ref{OP}), which are then applied to $C_t$ to show that it is both power bounded and uniformly mean ergodic in all spaces $X\not=\omega$; see Theorem \ref{Ct-PM}. The same is true for $C_t$ acting in $\omega$; see Theorem \ref{Dyn-omega}. In this section we also investigate the properties of the dual operators $C'_t$ acting in $X'_\beta$, which are given by \eqref{DualC} and \eqref{eq.DualeO-X}. The operators $C_t'$ are compact and their spectra are identified in Proposition \ref{PP-dualOperator}, where it is also shown that the set of all eigenvalues of $C'_t$ is $\Lambda$. Moreover, for each $n\in\N_0$, the eigenvector in $X'_\beta$ spanning the $1$-dimensional eigenspace corresponding to $\frac{1}{n+1}\in\Lambda$ is also determined. A consequence of $C'_t$ having a rich supply of eigenvalues is  that each operator $C_t\colon X\to X$, for $t\in [0,1)$, fails to be supercyclic. Moreover, it is established in Proposition \ref{P-6-8} that $C'_t\colon X'_\beta\to X'_\beta$ is power bounded, uniformly mean ergodic but, not supercyclic. It should be noted that the main results in this section are also new for $C_t$ acting in the Banach spaces $\ell^p$, $ces(p)$ and $d_p$.

\markboth{A.\,A. Albanese, J. Bonet and W.\,J. Ricker}%
{\MakeUppercase{Spectral properties of  generalized Ces\`aro operators}}

\section{Preliminaries}

Given locally convex Haudorff spaces $X, Y$ (briefly, lcHs) we denote by $\cL(X,Y)$ the space of all continuous linear operators from $X$ into $Y$. If $X=Y$, then we simply write $\cL(X)$ for $\cL(X,X)$. Equipped with the topology  of pointwise convergence  $\tau_s$ on $X$ (i.e., the strong operator topology) the lcHs $\cL(X)$ is denoted by $\cL_s(X)$ and for the topology $\tau_b$ of uniform convergence on bounded sets the lcHs $\cL(X)$ is denoted by $\cL_b(X)$.
Denote by $\cB(X)$ the collection of all bounded subsets of $X$ and by $\Gamma_X$ a system of continuous seminorms determing the topology of $X$.
 The identity operator on $X$ is denoted by $I$.  The \textit{dual operator} of $T\in \cL(X)$ is denoted by  $T'$; it acts in the topological dual space $X':=\cL(X,\C)$ of $X$. Denote by $X'_\si$ (resp., by $X'_\beta$) the space $X'$ with the weak* topology $\si(X',X)$ (resp., with the strong topology $\beta(X',X)$); see \cite[\S 21.2]{23} for the definition. It is known that $T'\in \cL(X'_\si)$ and $T'\in \cL(X_\be')$,  \cite[p.134]{24}.
For the general theory of functional analysis and operator theory relevant to this paper see, for example, \cite{Ed,Gr,J,MV,PB,Wa}.

\begin{lemma}\label{Iso} Let $X$ be a  lcHs and $T\in \cL(X)$ be an isomorphism of $X$ onto itself. Then $T'$ is an isomorphism of $X'_\beta$ onto itself. If, in addition, $X$ is  complete and barrelled, then $T$ is an isomorphism of $X$ onto itself if, and only if, $T'$ is an isomorphism of $X'_\beta$ onto itself.
\end{lemma}

\begin{proof} If $T$ is an isomorphism of $X$ onto itself, then 	 $T^{-1}\in \cL(X)$ exists with $TT^{-1}=T^{-1}T=I$. It was already noted that $T', (T^{-1})'\in \cL(X'_\beta)$ and clearly $(T^{-1})'T'=T'(T^{-1})'=I$.
	Thus, $(T')^{-1}$ exists in $\cL(X'_\beta)$ and  $(T')^{-1}=(T^{-1})'$; that is,
	$T'$ is an isomorphism of $X_\beta'$ onto itself.
	
	Suppose that $X$ is also complete and  barrelled and that $T'\in \cL(X'_\beta)$ is an isomorphism of $X'_\beta$ onto itself. As proved above, $T''$ is necessarily an isomorphism of $X''_\beta$ onto itself. By the proof of  Lemma 3 in \cite{ABR-00} it follows that $T$ is an isomorphism of $X$ onto itself. This completes
	the proof.
\end{proof}

Given a lcHs $X$ and $T\in \cL(X)$, the resolvent set $\rho(T;X)$ of $T$ consists of all $\lambda\in\C$ such that $R(\lambda,T):=(\lambda I-T)^{-1}$ exists in $\cL(X)$. The set $\sigma(T;X):=\C\setminus \rho(T;X)$ is called the \textit{spectrum} of $T$. The \textit{point spectrum}  $\sigma_{pt}(T;X)$ of $T$ consists of all $\lambda\in\C$ (also called an eigenvalue  of $T$) such that $(\lambda I-T)$ is not injective. An eigenvalue $\lambda$ of $T$ is called \textit{simple} if ${\rm dim} \Ker (\lambda I-T)=1$. Some authors (eg. \cite{Wa}) prefer the subset $\rho^*(T;X)$ of $\rho(T;X)$ consisting of all $\lambda\in\C$ for which there exists $\delta>0$ such that the open disc $B(\lambda,\delta):=\{z\in\C:\, |z-\lambda|<\delta\}\su \rho(T;X)$ and $\{R(\mu,T):\, \mu\in B(\lambda,\delta)\}$ is an equicontinuous subset of $\cL(X)$. Define $\sigma^*(T;X):=\C\setminus \rho^*(T;X)$, which is a closed set with $\sigma(T;X)\su \sigma^*(T;X)$. If $X$ is a Banach space, then $\sigma(T;X)= \sigma^*(T;X)$. For the spectral theory of compact operators in lcHs' we refer to \cite{Ed,Gr}, for example.

\begin{corollary}\label{SpettroDuale} Let $X$ be a complete, barrelled lcHs and $T\in \cL(X)$. Then \begin{equation}\label{eq.DD-spettro}
		\rho(T;X)=\rho(T';X'_\beta)\ \mbox{ and }\   \sigma(T;X)=\sigma(T';X'_\beta).
\end{equation}
Moreover, 
\begin{equation}\label{eq.DD-Spettro}
	\sigma^*(T'; X'_\beta)\subseteq \sigma^*(T;X).
\end{equation}
\end{corollary}

\begin{proof} The identities in \eqref{eq.DD-spettro} are an immediate consequence of 
Lemma \ref{Iso}.

Fix $\lambda\in \rho^*(T; X)$. Then there exists $\delta>0$ such that $B(\lambda, \delta)\subseteq \rho(T;X)$ and $\{R(\mu;T)\,:\, \mu\in B(\lambda,\delta)\}\subseteq \cL(X)$ is equicontinuous. For each $\mu\in B(\lambda,\delta)$ it follows from the proof of Lemma \ref{Iso} that $R(\mu, T)'=((\mu I-T)^{-1})'=(\mu I-T')^{-1}=R(\mu, T')$. Then \cite[\S 39.3(6), p.138]{24} implies that $\{R(\mu, T')\, :\, \mu\in B(\lambda,\delta)\}\subseteq \cL(X'_\beta)$ is equicontinuous, that is, $\lambda\in \rho^*(T';X'_\beta)$. So, we have established that  $\rho^*(T;X)\subseteq \rho^*(T';X'_\beta)$; taking complements yields \eqref{eq.DD-Spettro}. 
\end{proof}

 A linear map $T\colon X\to Y$, with $X,Y$ lcHs', is called \textit{compact} if there exists a neighbourhood $\cU$ of $0$ in $X$ such that $T(\cU)$ is a relatively compact set in $Y$. It is routine to show that necessarily $T\in \cL(X,Y)$. For the following result see \cite[\S 42.1(1)]{24} or \cite[Proposition 17.1.1]{J}.

 \begin{lemma}\label{Comp-Op} Let $X$ be a lcHs. The compact operators are a 2-sided ideal in $\cL(X)$.\end{lemma}

 To establish the continuity of	$C_t$, for $t\in [0,1]$, in the Fr\'echet spaces considered in this paper we will need the following result, \cite[Lemma 25]{ABR8}.

 \begin{lemma}\label{Cont-F}
 	 Let $X=\cap_{n=1}^\infty X_n$ and $Y=\cap_{m=1}^\infty Y_m$ be two Fr\'echet spaces which resp. are  the intersection of the sequence of Banach spaces $(X_n,\|\cdot\|_n)$, for $n\in\N$, and of the sequence of Banach spaces $(Y_m,|||\cdot|||_m)$, for $m\in\N$, satisfying $X_{n+1}\subset X_n$ with $\|x\|_n\leq \|x\|_{n+1}$ for each $n\in\N$ and $x\in X_{n+1}$ and $Y_{m+1}\subset Y_m$ with $|||y|||_m\leq |||y|||_{m+1}$ for each $m\in\N$ and $y\in Y_{m+1}$. Suppose that $X$ is dense in $X_n$ for each $n\in\N$. Then a linear operator $T\colon X\to Y$ is continuous if, and only if, for each $m\in\N$ there exists $n\in\N$ such that the operator $T$ has a unique continuous extension $T_{n,m}\colon X_n\to Y_m$.
 \end{lemma}

 The following result, based on \cite[Lemma 2.1]{ABR3}, will be needed to determine the spectra of $C_t$, for $t\in [0,1]$, in the Fr\'echet spaces considered in this paper.

 \begin{lemma}\label{Sp-FF} Let $X=\cap_{n=1}^\infty X_n$ be a Fr\'echet space which is the intersection of a sequence of Banach spaces $(X_n,\|\cdot\|_n)$, for $n\in\N$, satisfying $X_{n+1}\subset X_n$ with $\|x\|_n\leq \|x\|_{n+1}$ for each $n\in\N$ and $x\in X_{n+1}$. Let $T\in \cL(X)$ satisfy the following condition:
 	
 	(A) For each $n\in\N$ there exists $T_n\in \cL(X_n)$ such that the restriction of $T_n$ to $X$ (resp. of $T_n$ to $X_{n+1}$) coincides with $T$ (resp. with $T_{n+1}$).
 	
 	Then the following properties are satisfied.
 	\begin{itemize}
 		\item[\rm (i)] $\sigma(T;X)\su \cup_{n=1}^\infty \sigma(T_n;X_n)$ and $\sigma_{pt}(T;X)\su \cap_{n=1}^\infty \sigma_{pt}(T_n;X_n)$. 
 		\item[\rm (ii)] If $\cup_{n=1}^\infty \sigma(T_n;X_n)\su\overline{\sigma(T;X)}$, then $\sigma^*(T;X)=\overline{\sigma(T;X)}$.
 		 \item[\rm (iii)] If ${\rm dim}\ker(\lambda I-T_m)=1$ for each $\lambda \in \cap_{n=1}^\infty \sigma_{pt}(T_n;X_n)$ and $m\in\N$, then $\sigma_{pt}(T;X)=\cap_{n=1}^\infty \sigma_{pt}(T_n;X_n)$.
 	\end{itemize}
 	 \end{lemma}

  \begin{proof} In view of \cite[Lemma 2.1]{ABR3} it remains to show the validity of the inclusion $\sigma_{pt}(T;X)\su\cap_{n=1}^\infty \sigma_{pt}(T_n,X_n)$ in the statement (i) and the identity in (iii).

  	The inclusion  $\sigma_{pt}(T;X)\su \cap_{n=1}^\infty \sigma_{pt}(T_n;X_n)$ is clear. Indeed, if $(\lambda I-T)x=0$ for some $x\in X\setminus\{0\}$ and $\lambda\in\C$, then in view of $X\subseteq X_n$ and $T_n|_X=T$, for $n\in\N$, (see condition (A)), we have that $x\in X_n\setminus\{0\}$ and $(\lambda I-T_n)x=0$ for every $n\in\N$. Hence, $\lambda \in \cap_{n=1}^\infty \sigma_{pt}(T_n,X_n)$.
  	
  	To establish the validity of (iii), fix $\lambda \in \cap_{n=1}^\infty \sigma_{pt}(T_n,X_n)$. Then, for each $n\in\N$, there exists $x_n\in X_n\setminus\{0\}$ such that $(\lambda I-T_n)x_n=0$. Since $x_{n+1}\in X_{n+1}\su X_n$, for $n\in\N$, condition (A)  implies that  also  $(\lambda I-T_n)x_{n+1}=0$ in $X_n$ for each $n\in\N$. So, for each $n\in\N$, we have that $x_{n+1}=\mu_n x_n$ for some $\mu_n\in \C\setminus\{0\}$. Therefore, $x_n=(\prod_{j=1}^{n-1}\mu_j)x_1$, with $\prod_{j=1}^{n-1}\mu_j\not =0$. Accordingly, $x_1\in X_n$ for each $n\in\N$ and hence, $x_1\in X$. On the other hand, applying again condition (A), we can conclude that $(\lambda I-T)x_1=(\lambda I-T_1)x_1=0$, i.e., $\lambda\in \sigma_{pt}(T;X)$.
  	\end{proof}

 Fr\'echet spaces $X$ which satisfy the assumptions of Lemma \ref{Sp-FF} are often called \textit{countably normed Fr\'echet spaces}; for the general theory of such spaces see \cite{G-S2}, for example.

A Hausdorff locally convex space $(X,\tau)$ is called an \textit{(LB)-space} if there is a sequence $(X_k)_{k\in\N}$ of Banach spaces satisfying $X_k\su X_{k+1}$ continuously for $k\in\N$,   $X=\cup_{k=1}^\infty X_k$ and $\tau$ is the finest locally convex topology on $X$ such that the natural inclusion $X_k\subset X$ is continuous for each $k\in\N$, \cite[pp.290-291]{MV}. In this case we write $X= \ind_k X_k$.
  If, in addition, $X$ is a \textit{regular} (LB)-space, \cite[p.83]{J}, then a set $B\subset X$ is bounded if and only if there exists $m\in\N$ such that $B\subset X_m$ and $B$ is bounded in the Banach space $X_m$. Complete (LB)-spaces are regular, \cite[\S 19.5(5)]{23}. All of the (LB)-spaces of sequences  considered in this note will be regular because of the following result, \cite[Proposition 25.19(2)]{MV}.

 \begin{lemma}\label{Reg-LB} Let $X=\ind_k X_k$ be an (LB)-space with an increasing union of reflexive Banach spaces $X=\cup_{k=1}^\infty X_k$ such that each inclusion $X_k\su X_{k+1}$, for $k\in\N$, is continuous. Then $X$ is complete and hence, also regular.
 \end{lemma}

An (LB)-space $X=\ind_k X_k$ is said to be \textit{boundedly retractive} if for every $B\in \cB(X)$ there exists $k\in\N$
 such that $B$ is contained  and bounded in $X_k$, and $X$ and $X_k$ induce the same topology on
$B$. The (LB)-space $X$ is said to be \textit{sequentially retractive} if for every null sequence in $X$ there exists $k\in\N$ such that the sequence is contained and converges to zero in $X_k$. Finally, the (LB)-space $X$ is
said to be \textit{compactly regular} if for every  compact subset $C$ of $X$ there exists $k\in\N$ such that $C$ is compact in $X_k$.    Each of these three notions implies the completeness of $X$, \cite[Corollary 2.8]{W}. Neus \cite{N} proved that all these notions are equivalent even for inductive limits of normed spaces.

 In the setting of boundedly retractive (LB)-spaces, the following general statement on the compactness of certain dual operators is valid.

\begin{prop}\label{Comp-dual} Let $X$ be a lcHs,  $Y=\ind_k Y_k$ be a boundedly retractive (LB)-space and $T\in \cL(X,Y)$ be compact. Then $T'\in \cL(Y'_\beta, X'_\beta)$ is compact.
\end{prop}

\begin{proof}
The compactness of $T$ implies that there exists a closed,  absolutely convex neighbourhood $\cU$ of $0$ in $X$ such that $T(\cU)$ is a relatively compact set in $Y$. So, the closure $B:=\overline{T(\cU)}\in \cB(Y)$ of $T(\cU)$ is a compact set in $Y$. But,  $Y$ is a boundedly retractive (LB)-space. Accordingly, there exists $k\in\N$ such that $B$ is contained  and bounded in $Y_k$, and $Y$ and $Y_k$ induce the same topology on $B$. Therefore, $B$ is also a compact set in $Y_k$ and $T(X)\subseteq Y_k$.  Accordingly, the operator $T$ acts compactly from $X$ into $Y_k$. Denote by $T_1$ the operator $T$ when interpreted to be acting from $X$ into $Y_k$ and by $i_k$ the continuous inclusion of $Y_k$ into $Y$. So, $T_1\in\cL(X,Y_k)$ is compact and $T=i_kT_1$. Denote by $p$ the continuous seminorm on $X$ corresponding to $\cU$  and let  $X_p$ denote the normed quotient space  $\left(\frac{X}{\Ker p}, p\right)$. Then there exists a unique continuous linear operator $S$ from $X_p$ into $Y_k$ such that $SQ=T_1$, where $Q$ denotes the canonical quotient map from $X$ into $X_p$ and hence, is an open map. Since $T_1\in \cL(X,Y_k)$ is  compact and $Q\in \cL(X,X_p)$ is open, the operator $S\in \cL(X_p, Y_k)$ is necessarily compact. By Schauder's theorem,  \cite[\S 42(7), p.202]{24}, it follows that $S'\in\cL(Y'_k, X_p')$ is compact. So, $T'_1=Q'S'\in \cL(Y'_k,X'_\beta)$ is compact  and hence, $T'=T'_1i_k'\in \cL(Y'_\beta, X'_\beta)$ is compact (cf. Proposition 17.1.1 in \cite{J}). This completes the proof.
\end{proof}

A Fr\'echet space $X$ is said to be \textit{quasinormable} if for every neighbourhhod $\cU$ of $0$ in $X$ there exists a neighbourhhod $\cV$ of $0$ in $X$ so that, for every $\varepsilon> 0$, there exists $B\in \cB(X)$ satisfying $\cV\subseteq B+\varepsilon \cU$. Thus, every  Fr\'echet-Schwartz space is quasinormable \cite[Remark, p.313]{MV}. The strong dual $X'_\beta$ of a quasinormable Fr\'echet space $X$ is necessarily a boundedly retractive (LB)-space \cite[Theorem]{Bo}. Thus, the strong dual of any Fr\'echet-Schwartz  space (briefly, (DFS)-space) is a boundedly retractive (LB)-space.

\begin{corollary}\label{C-nuovo} Let $X$ and $Y$ be two Fr\'echet spaces and  $T\in \cL(X,Y)$. If $T''\in \cL(X''_\beta, Y''_\beta)$ is compact, then $T$ is compact.
	
	If, in addition,  $X$ is quasinormable and  $T$ is compact, then $T''\in \cL(X''_\beta, Y''_\beta)$ is compact.
	\end{corollary} 

\begin{proof} Suppose that $T''\in \cL(X''_\beta, Y''_\beta)$ is compact. Since $X$, $Y$ are Fr\'echet spaces, they are isomorphic to their respective natural image in $X''_\beta$, $Y''_\beta$ (in which they are closed subspaces). Moreover, the restriction of $T''$ to $X$ coincides with $T$ and takes its values in $Y\subseteq Y''_\beta$. Then the compactness of $T$ follows from that of $T''$.
	
	Suppose that $X$ is quasinormable and that $T\in \cL(X,Y)$ is compact. Since $X$ is quasinormable, its strong dual $X'_\beta$ is a boundedly retractive (LB)-space. Moreover, $Y$ being a Fr\'echet space implies that $T'\colon Y'_\beta\to X'_\beta$ is compact, \cite[Corollary 9.6.3]{Ed}. It follows from Proposition \ref{Comp-dual}, with $Y'_\beta$ in place of $X$ and $X'_\beta$ in place of $Y=\ind_k Y_k$ and $T'$ in place of $T$, that $T''\in \cL(X''_\beta, Y''_\beta)$ is compact. 
\end{proof}

To identify the spectrum of $C_t$ acting in the (LB)-spaces arising in this paper we will require the following two results; the first one, i.e.  Lemma \ref{LB-Op}, is a direct consequence of Grothendieck's factorization theorem (see e.g. \cite[Theorem 24.33]{MV}), and the second one, i.e. Lemma \ref{Sp-LB}, is proved in \cite[Lemma 5.2]{ABR5}.

\begin{lemma}\label{LB-Op} Let $X=\ind_n X_n$ and $Y=\ind_mY_m$ be two (LB)-spaces with  increasing unions of  Banach spaces $X=\cup_{n=1}^\infty X_n$ and $Y=\cup_{m=1}^\infty Y_m$. Let $T\colon X\to Y$ be a linear map. Then $T$ is continuous (i.e., $T\in \cL(X,Y)$) if and only if for each $n\in\N$ there exists $m\in\N$ such that $T(X_n)\su Y_m$ and the restriction $T\colon X_n\to Y_m$ is continuous.
	\end{lemma}

\begin{lemma}\label{Sp-LB} Let $X=\ind_k X_k$ be a Hausdorff inductive limit of a sequence of Banach spaces $(X_k,\|\cdot\|_k)$. Let $T\in \cL(X)$ satisfy the following condition:
	
	(A\'\,) For each $k\in\N$ the restriction $T_k$ of $T$ to $X_k$ maps $X_k$ into itself and $T_k\in \cL(X_k)$.
	
	Then the following properties are satisfied.
	\begin{itemize}
		\item[\rm (i)] $\sigma_{pt}(T;X)=\cup_{k=1}^\infty\sigma_{pt}(T_k,X_k)$.
		\item[\rm (ii)] If $\cup_{k=m}^\infty\sigma(T_k;X_k)\su \overline{\sigma(T;X)}$ for some $m\in\N$, then $\sigma^*(T;X)=\overline{\sigma(T;X)}$.
		\item[\rm (iii)] $\sigma(T;X)\su \cap_{m\in\N}\left(\cup_{n=m}^\infty\sigma(T_n;X_n)\right)$.
	\end{itemize}
	\end{lemma}

Another useful fact for our study is the following result.

\begin{lemma}\label{C-incl} Let $T\in \cL(\omega)$. Let $X$ be a Fr\'echet space or an (LB)-space continuously included in $\omega$. If $T(X)\su X$, then $T\in \cL(X)$.
\end{lemma}

\begin{proof} The result follows from the closed graph theorem, \cite[Theorem 24.31]{MV}, after recalling that $X$ is ultrabornological, \cite[Remark 24.15(c) \& Proposition 24.16]{MV} and has a web, \cite[Corollary 24.29 \& Remark 24.36]{MV}. So, it is enough to show that the graph of $T$ in $X$ is closed. To do this, we assume that a net $(x_\alpha)_\alpha\su X$ satisfies $x_\alpha \to x$ and $T(x_\alpha)\to y$ in $X$. Since the inclusion $X\su \omega$ is continuous, $x_\alpha\to x$ in $\omega$ and hence, $T(x_\alpha)\to T(x)$ in $\omega$. On the other hand, by the continuity of the inclusion  $X\su \omega$ also $T(x_\alpha)\to y$ in $\omega$. Then $T(x)=y$. So, $(x,y)$ belongs to the graph of $T$. This shows that  the graph of $T$ is closed.
\end{proof}

For $X$ a barrelled lcHs, every bounded subset of $\cL_s(X)$  is equicontinuous, \cite[Proposition 23.27]{MV}. It is known that every Fr\'echet space is barrelled, \cite[Remark, p.296]{MV}, and that every (LB)-space is barrelled, \cite[Proosition 24.16]{MV}.

The operator norm of a Banach space operator $T\in \cL(X,Y)$ will be denoted by $\|T\|_{X\to Y}$. The Banach spaces $\ell^p=\ell^p(\N_0)$, for $1\leq p<\infty$, with their standard norm $\|\cdot\|_p$ are classical. For $1<p<\infty$ these spaces are reflexive. The spectra of $C_t$ acting in such  spaces are given in the following result; see \cite{YD} for $1<p<\infty$ and also \cite[\S 8]{SEl-S} for $1\leq p<\infty$. Recall from Section 1 that 
$$
\Lambda:=\left\{\frac{1}{n+1}:\, n\in\N_0\right\}.
$$

\begin{prop}\label{Sp-spazilp} For each $t\in [0,1)$ the operator $C_t\in \cL(\ell^p)$, for $1\leq p<\infty$, is a compact operator satisfying
	\[
	\|C_t\|_{\ell^1\to\ell^1}=\frac{1}{t}\log\left(\frac{1}{1-t}\right),\quad t\in (0,1),
	\]
and
\[
\left(\sum_{n=0}^\infty \left(\frac{t^n}{n+1}\right)^p\right)^{1/p}\leq 	\|C_t\|_{\ell^p\to\ell^p}\leq \left(\frac{1}{t}\log\left(\frac{1}{1-t}\right)\right)^{1/p},\ 1<p<\infty,\ t\in (0,1),
\]	
with $\|C_0\|_{\ell^p\to\ell^p}=1$. Moreover,
\begin{equation}\label{Sp-lp}
	\sigma_{pt}(C_t;\ell^p)=\Lambda\ \hbox{ and }\ \sigma(C_t; \ell^p)=\Lambda\cup\{0\}.
	\end{equation}
	\end{prop}

Concerning the classical Cesàro operator $C_1$ (c.f. \eqref{Ces-1}) in $\cL(\ell^p)$ we have the following result.

\begin{prop}\label{Sp-C1-lp} Let $1<p<\infty$. 
	\begin{itemize}
		\item[\rm (i)] The operator $C_1\in \cL(\ell^p)$ with $\|C_1\|_{\ell^p\to\ell^p}=p'$, where $\frac{1}{p}+\frac{1}{p'}=1$.
		\item[\rm (ii)] The spectra of $C_1$ are given by 
		\[
		\sigma_{pt}(C_1;\ell^p)=\emptyset \ \mbox{ and } \ \sigma(C_1;\ell^p)=\left\{z\in\C\,:\, \left|z-\frac{p'}{2}\right|\leq \frac{p'}{2}\right\}.
		\]
	\end{itemize}
Moreover, the range $(C_1-zI)(\ell^p)$ is not dense in $\ell^p$ whenever $|z-\frac{p'}{2}|<\frac{p'}{2}$.
	\end{prop}

For part (i) we refer to \cite[Theorem 326]{HLP} and for part (ii) see \cite{G}, \cite{Le}, \cite{Rho} and the references therein. In particular, $C_1$ is a \textit{not} a compact operator.

G. Bennett thoroughly investigated the discrete Cesàro spaces
\[
ces(p):=\{x\in\omega:\, C_1|x|\in\ell^p\},\quad 1<p<\infty,
\]
where $|x|:=(|x_n|)_{n\in\N_0}$, which satisfy $\ell^p\subseteq ces(p)$ continuously and are reflexive Banach spaces relative to the norm
\begin{equation}\label{eq.norm-ces}
	\|x\|_{ces(p)}:=\|C_1|x|\|_p,\quad x\in ces(p);
\end{equation}
see, for example, \cite{Be}, as well as \cite{AM,CR3,G-E,L-M} and the references therein. The following result, \cite[Proposition 5.6]{CR4} describes the spectra of $C_t$ acting in $ces(p)$.

\begin{prop}\label{Sp-spazices} Let $t\in [0,1)$ and $1<p<\infty$. The operator $C_t\in \cL(ces(p))$ is compact and satisfies
	\[
	\|C_t\|_{ces(p)\to ces(p)}\leq \min\left\{\frac{1}{1-t},\frac{p}{p-1}\right\}.
	\]
Moreover,
\begin{equation}\label{Sp-ces}
	\sigma_{pt}(C_t;ces(p))=\Lambda\ \hbox{ and }\ \sigma(C_t; ces(p))=\Lambda\cup\{0\}.
	\end{equation}
\end{prop}

The situation for $C_1\in \cL(ces(p))$ is quite different. Indeed, $\|C_1\|_{ces(p)\to ces(p)}=p'$ and the spectra are given by
\[
\sigma_{pt}(C_1;ces(p))=\emptyset\ \mbox{ and }\ \sigma(C_1; ces(p))=\left\{z\in \C\,:\, \left|z-\frac{p'}{2}\right|\leq \frac{p'}{2}\right\}
\]
for each $1<p<\infty$; see Theorem 5.1 and its proof in \cite{CR3}. In particular, $C_1$ is \textit{not} a compact operator.

The dual Banach spaces $(ces(p))'$, for $1<p<\infty$, are rather complicated, \cite{Ja}. A more transparent \textit{isomorphic} identification of $(ces(p))'$ is given in Corollary 12.17 of \cite{Be}. It is shown there that
\[
d_p:=\{x\in\ell^\infty:\, \hat{x}:=(\sup_{k\geq n}|x_k|)_{n\in\N_0}\in \ell^p\},\quad 1<p<\infty,
\]
is a Banach space for the norm
\begin{equation}\label{norma}
	\|x\|_{d_p}:=\|\hat{x}\|_p,\quad x\in d_p,
\end{equation}
which is isomorphic to $(ces(p'))'$, where $p'$ is the conjugate exponent of $p$. The sequence $\hat{x}$ is called the \textit{least decreasing majorant} of $x$. The duality is the natural one given by
\[
\langle w,x\rangle :=\sum_{n=0}^\infty w_nx_n,\quad w\in ces(p'),\ x\in d_p.
\]
In particular, $d_p$ is reflexive for each $1<p<\infty$. Since $|x|\leq |\hat{x}|$, it is clear that $\|x\|_p\leq \|\hat{x}\|_p=\|x\|_{d_p}$, for $x\in d_p$, that is, $d_p\subseteq \ell^p$ continuously. So, for all $1<p<\infty$, we have $d_p\subseteq \ell^p\subseteq ces(p)$ with continuous inclusions.
 The following result is Theorem 6.9 of \cite{CR4}.

\begin{prop}\label{Sp-spazidp}
	Let $t\in [0,1)$ and $1<p<\infty$. The operator $C_t\in \cL(d_p)$ is compact and satisfies
	\[
	\|C_t\|_{d_p\to d_p}\leq (1-t)^{-1-(1/p)}.
	\]
	Moreover,
	\begin{equation}\label{Sp-dp}
		\sigma_{pt}(C_t;d_p)=\Lambda\ \hbox{ and }\ \sigma(C_t; d_p)=\Lambda\cup\{0\}.
	\end{equation}
\end{prop}

Concerning the operator $C_1\in \cL(d_p)$, $1<p<\infty$, it is known that $\|C_1\|_{d_p\to d_p}=p'$ and that its spectra are given by
\[
\sigma_{pt}(C_1;d_p)=\emptyset\ \mbox{ and }\ \sigma(C_1; d_p)=\left\{z\in\C\,:\, \left|z-\frac{p'}{2}\right|\leq \frac{p'}{2}\right\};
\] 
see Proposition 3.2 and Corollary 3.5 in \cite{BR1}.

\section{The operators $C_t$ acting in $\omega$}

Given an element $x=(x_n)_{n\in\N_0}\in \omega$ we write $x\geq 0$ if $x=|x|=(|x_n|)_{n\in\N_0}$. By $x\leq z$ it is meant that $(z-x)\geq 0$. The sequence space $\omega$ is a non-normable Fr\'echet space for the Hausdorff locally convex topology of coordinatewise convergence, which is determined by the increasing sequence of seminorms
\begin{equation}\label{sem-omega}
	r_n(x):=\max_{0\leq j\leq n}|x_j|,\quad x\in\omega,
	\end{equation}
for each $n\in\N_0$. Observe that $r_n(x)=r_n(|x|)\leq r_n(|y|)=r_n(y)$ whenever $x,y\in\omega$ satisfy $|x|\leq |y|$. Let $e_n:=(\delta_{nj})_{j\in\N_0}$ for  each $n\in\N_0$ and set $\cE:=\{e_n:\, n\in\N_0\}$. It is clear from \eqref{Ces-op} that each $C_t\colon\omega\to\omega$ is a linear map which is represented by a lower triangular matrix with respect to the unconditional basis $\cE$ of $\omega$. Namely,
\begin{equation}\label{Matrix}
	C_t\simeq \left(\begin{array}{ccccc}
		1 & 0 & 0 & 0 & \cdots \\
		t/2 & 1/2 & 0 & 0 & \cdots \\
		t^2/3 & t/3 & 1/3 & 0 & \cdots \\
		t^3/4 & t^2/4 & t/4 & 1/4 & \cdots \\
		\vdots & \vdots & \vdots & \vdots & \\
	\end{array}\right)
\end{equation}
with main diagonal the positive, decreasing sequence given by
\begin{equation}\label{dia}
\varphi:=\left(\frac{1}{n+1}\right)_{n\in\N_0}\in c_0.
\end{equation}
The following properties of $C_t$ are recorded in \cite[Lemma 2.1]{CR4}, except for part (iv).

\begin{lemma}\label{L-3-1} Let $t\in [0,1)$.
	\begin{itemize}
		\item[\rm (i)] Each $C_t$ is a positive operator on $\omega$ ,i.e., $C_tx\geq 0$ whenever $x\geq 0$.
		\item[\rm (ii)] Let $0\leq r\leq s\leq 1$. Then
		\[
		0\leq |C_rx|\leq C_r|x|\leq C_s|x|,\quad x\in\omega.
		\]
		\item[\rm (iii)] For each $t\in [0,1)$ the identities
		\[
		C_te_n=\sum_{k=0}^\infty \frac{t^k}{k+n+1}e_{k+n}\in \ell^1, \quad n\in\N_0,
		\]
		and
		\[
		C_t(e_n-te_{n+1})=\frac{1}{n+1}e_n,\quad n\in\N_0,
		\]
		are valid.
		\item[\rm (iv)] For each $1<q<\infty$ we have $d_q\subseteq \ell^q\subseteq ces(q)$ with continuous inclusions.
	\end{itemize}
	\end{lemma}

\begin{proof}
	(iv) In view of the discussion after \eqref{norma} it remains to establish that $ces(q)\subseteq \omega$ continuously. Fix $x\in ces(q)$. Given $n\in\N_0$ observe that
	\[
	|x_k|\leq (n+1)\frac{|x_0|+|x_1|+\ldots +|x_n|}{n+1}\leq (n+1)\|C_1|x|\|_q=(n+1)\|x\|_{ces(q)}, \ 0\leq k\leq n.
	\]
	It follows from \eqref{sem-omega} that $r_n(x)\leq (n+1)\|x\|_{ces(q)}$. Since $n\in\N_0$ is arbitrary, we can conclude that $ces(q)\subseteq \omega$ continuously.
\end{proof}

The classical Cesàro operator $C_1\colon\omega\to\omega$ is a bicontinuous topological isomorphism (and hence, is \textit{not} a compact operator) with spectra given by
\[
\sigma(C_1;\omega)=\sigma_{pt}(C_1;\omega)=\Lambda \ \hbox{ and }\ \sigma^*(C_1;\omega)=\Lambda\cup\{0\};
\]
see \cite[p.285 \& Proposition 4.4]{ABR3}. So, we will only consider the case $t\in [0,1)$.

Let $t\in [0,1)$ and fix $n\in\N_0$. According to \eqref{Ces-op} and \eqref{sem-omega}, for each $x\in \omega$, it is the case that
\begin{equation}\label{eq.ContS-omega}
r_n(C_tx)=\max_{0\leq k\leq n}\left|\frac{1}{k+1}\sum_{i=0}^{k-1}t^{k-i}x_i\right|\leq \max_{0\leq k\leq n}\frac{1}{k+1}\sum_{i=0}^{k-1}|x_i|\leq r_n(x).
\end{equation}
This implies that $C_t\in \cL(\omega)$ and that the family of operators $\{C_t:\, t\in [0,1)\}$ is an equicontinuous subset of $\cL(\omega)$.

\begin{prop}\label{nocompact-omega} For each $t\in [0,1)$ the operator $C_t\in \cL(\omega)$ is a bicontinuous isomorphism of $\omega$ onto itself with inverse operator $(C_t)^{-1}\colon\omega\to\omega$ given by
	\begin{equation}\label{inverse}
		(C_t)^{-1}y=((n+1)y_n-nty_{n-1})_{n\in\N_0},\quad y\in\omega \ (\hbox{with}\ y_{-1}:=0).
	\end{equation}
In particular, $C_t$ is not a compact operator.
	\end{prop}

\begin{proof} Fix $t\in [0,1)$. Let $x\in\omega$ satisfy $C_tx=0$. Considering the coordinate $0$ of $C_tx=0$ yields $x_0=0$; see \eqref{Ces-op}. The equation for coordinate $1$ of $C_tx=0$ is $\frac{tx_0+x_1}{2}=0$ (cf. \eqref{Ces-op}) which yields $x_1=0$. Proceed inductively for successive coordinates reveals that $x_n=0$ for all $n\in\N_0$. Hence, $C_t$ is injective.
	
	Given $y\in \omega$ let $x\in\omega$ be the element on the right-side of \eqref{inverse}. Direct calculation shows that $C_tx=y$. Accordingly, $C_t$ is surjective.
	
	By the open mapping theorem for Fr\'echet spaces (cf. Corollary 24.29 and Theorem 24.30 in \cite{MV}) the operator $C_t$ is a bicontinuous isomorphism.
	
	Since $C_t$ is a bicontinuous isomorphism of $\omega$, which is an infinite dimensional Fr\'echet space, $C_t$ cannot be a compact operator.
	\end{proof}

To determine the spectrum of $C_t\in \cL(\omega)$ requires some preparation. Define 
\begin{equation}\label{set}
	\cS:=\left\{x\in\omega:\, \beta(x):=\lim_{n\to\infty}\frac{|x_{n+1}|}{|x_n|}<1\right\},
\end{equation}
with the understanding that there exists $N\in\N_0$ such that $x_n\not=0$ for $n\geq N$ and the limit $\beta(x)$ exists. Analogously to $d_p$, for $1<p<\infty$,  define
\begin{equation}\label{d1}
	d_1:=\{x\in\ell^\infty\,:\, \hat{x}:=(\sup_{k\geq n}|x_k|)_{n\in\N_0}\in \ell^1\};
	\end{equation}
see \cite{Be}, \cite{BR1}, \cite{CR3}, \cite{G-E} and the references therein. Then $d_1$ is a Banach lattice for the norm $\|x\|_{d_1}:=\|\hat{x}\|_1$ and the coordinatewise order. Since $0\leq |x|\leq \hat{x}$, for $x\in\ell^\infty$, it is clear that $\|x\|_1\leq \|x\|_{d_1}$ for $x\in d_1$, that is, $d_1\subseteq \ell^1$ with a continuous inclusion. 
Clearly, $d_1\subseteq d_p$, for all $1<p<\infty$, and   $d_1\subseteq \ell^1$ implies that $d_1\subseteq \ell^p$, for all $1<p<\infty$. Moreover,  $\ell^p\subseteq ces(p)$ (cf. Section 2) and so also $d_1\subseteq ces(p)$, for $1<p<\infty$. All inclusions are continuous. In view of Lemma \ref{L-3-1}(iv) it is clear that $d_1\subseteq \omega$ and $\ell^1\subseteq \omega$ continuously.
It is known that $\mathcal{S}\subseteq d_1$, \cite[Lemma 3.3]{CR4}. 

\begin{remark}\label{Nuova R}\rm Proposition \ref{Sp-spazidp} is also valid for $p=1$; see \cite[Theorem 6.9]{CR4}.
	\end{remark}

The following result, \cite[Lemma 3.6]{CR4}, will be required.

\begin{lemma}\label{L-3-3} Let $t\in [0,1)$ and $\varphi$ be as in \eqref{dia}. For each $m\in\N$ define $x^{[m]}\in\omega$ by
	\begin{equation}\label{eigenvalue}
		x^{[m]}:=\alpha_m \left(0,\ldots,0,1, \frac{(m+1)!}{m!\,1!}t,\frac{(m+2)!}{m!\,2!}t^2,\frac{(m+3)!}{m!\,3!}t^3,\ldots \right),
	\end{equation}
with $\alpha_m\in\C\setminus\{0\}$ arbitrary, where $1$ is in position $m$. For $m=0$ define $x^{[0]}:=\alpha_0(t^n)_{n\in\N_0}$ with $\alpha_0\in\C\setminus\{0\}$ arbitrary.
\begin{itemize}
	\item[\rm (i)] For each $m\in\N_0$, the vector $x^{[m]}$ is the unique solution in $\omega$ of the equation $C_tx=\varphi_m x=\frac{1}{m+1}x$ whose  $m$-th coordinate is $\alpha_m$.
	\item[\rm (ii)] The vector $x^{[m]}\in d_1\su\omega$, for each $m\in\N_0$.
\end{itemize}
	\end{lemma}

\begin{remark}\label{Point-Sp}\rm Let $t\in [0,1)$ and $X$ be any Banach space in $\{\ell^1,d_1\}\cup\{\ell^p, ces(p), d_p\,:\, 1<p<\infty\}$. For each $\nu\in \sigma_{pt}(C_t;X)=\Lambda$,  it is the case that ${\rm dim}\,\Ker(\nu I-C_t)=1$. Indeed, $d_1\subseteq X$; see the discussion prior to Remark \ref{Nuova R}. Given $\nu\in \Lambda$ there exists $m\in\N_0$ such that $\nu=\varphi_m$. According to Lemma \ref{L-3-3} the $1$-dimensional eigenspace corresponding to $\nu\in \sigma_{pt}(C_t;\omega)$ is spanned by $x^{[m]}$ with $x^{[m]}\in d_1$. The claim is thereby proved.	
\end{remark}

The next lemma places a restriction on where $\sigma(C_t;\omega)$ can be located in $\C$.

\begin{lemma}\label{L-3-4} Let $t\in [0,1)$. For each $\nu\in \C\setminus\Lambda$ the operator $C_t-\nu I$ is a bicontinuous isomorphism of $\omega$ onto itself. In particular, $\sigma(C_t;\omega)\subseteq \Lambda$.
	\end{lemma}

\begin{proof} Fix $\nu\not\in\Lambda$. Let $(C_t-\nu I)x=0$ for $x\in\omega$. It follows from \eqref{Matrix}, by equating the coordinate $0$ of $C_tx=\nu x$, that $x_0=\nu x_0$ and hence, as $\nu\not=1$, that $x_0=0$. Equating the coordinate $1$ of $C_tx=\nu x$ yields $\frac{tx_0+x_1}{2}=\nu x_1$. Since $x_0=0$ and $\nu\not=\frac{1}{2}$, it follows that $x_1=0$. Considering coordinate $2$ gives $\frac{t^2x_0+tx_1+x_2}{3}=\nu x_2$. Then  $x_0=x_1=0$ and $\nu\not=\frac{1}{3}$ imply $x_2=0$. Proceed inductively to conclude that $x=0$, that is, $C_t-\nu I$ is injective.
	
	To verify the surjectivity of $C_t-\nu I$ fix $y\in\omega$. It is required to show that there exists $x\in\omega$ satisfying $(C_t-\nu I)x=y$. Equating coordinate $0$ gives $x_0-\nu x_0=y_0$, that is, $x_0=y_0/(1-\nu)$. Considering coordinate $1$ yields $\frac{tx_0}{2}+(\frac{1}{2}-\nu)x_1=y_1$. Substituting for $x_0$ gives $(\frac{1}{2}-\nu)x_1=y_1-\frac{t}{2(1-\nu)}y_0$, that is, 
	\[
	x_1=\frac{y_1}{(\frac{1}{2}-\nu)}-\frac{ty_0}{2(\frac{1}{2}-\nu)(1-\nu)}.
	\]
	Next, an examination of coordinate $2$ yields $\frac{t^2}{3}x_0+\frac{t}{3}x_1+(\frac{1}{3}-\nu)x_2=y_2$. Substituting for $x_0$ and $x_1$ we can conclude that 
	\[
	x_2=\frac{y_2}{(\frac{1}{3}-\nu)}-\frac{ty_1}{3(\frac{1}{3}-\nu)(\frac{1}{2}-\nu)}+\frac{\nu t^2y_0}{3(\frac{1}{3}-\nu)(\frac{1}{2}-\nu)(1-\nu)}.
	\]
	Continuing inductively yields
	\begin{align}\label{eq.Inverse}
		x_n=&\frac{y_n}{(\frac{1}{n+1}-\nu)}-\frac{ty_{n-1}}{(n+1)(\frac{1}{n+1}-\nu)(\frac{1}{n}-\nu)}+\\
		&+\frac{\nu t^2y_{n-2}}{(n+1)(\frac{1}{n+1}-\nu)(\frac{1}{n}-\nu)(\frac{1}{n-1}-\nu)}\nonumber\\
	&-\frac{\nu^2 t^3y_{n-3}}{(n+1)(\frac{1}{n+1}-\nu)(\frac{1}{n}-\nu)(\frac{1}{n-1}-\nu)(\frac{1}{n-2}-\nu)}+ \ldots\nonumber\\
	& +(-1)^n \frac{\nu^{n-1} t^ny_{0}}{(n+1)(\frac{1}{n+1}-\nu)(\frac{1}{n}-\nu)\ldots (\frac{1}{2}-\nu)(1-\nu)}. \nonumber
		\end{align}
	Then $x\in\omega$ satisfies $(C_t-\nu I)x=y$. Hence, $C_t-\nu I$ is surjective.
\end{proof}

Combining the previous results yields the main result of this section.

\begin{theorem}\label{Sp-omega} For each $t\in [0,1)$ the spectra of $C_t\in \cL(\omega)$ are given by
	\[
	\sigma(C_t;\omega)=\sigma_{pt}(C_t;\omega)=\Lambda,
	\]
	with each eigenvalue being simple, and
	\[
	\sigma^*(C_t;\omega)=\Lambda\cup\{0\}.
	\]
	The $1$-dimensional eigenspace corresponding to the eigenvalue $1/(m+1)\in\Lambda$ is spanned by $x^{[m]}$ (cf. \eqref{eigenvalue}), for each $m\in\N_0$.
	\end{theorem}

\begin{proof} It is clear from Lemma \ref{L-3-3} that $\Lambda\subseteq \sigma_{pt}(C_t;\omega)$ and that each point  $1/(m+1)\in \Lambda$ is a simple eigenvalue of $C_t$, whose corresponding eigenspace is spanned by $x^{[m]}$, for each $m\in\N_0$. Since $\sigma(C_t;\omega)\subseteq \Lambda$ (cf. Lemma \ref{L-3-4}) and $\sigma_{pt}(C_t;\omega)\subseteq\sigma(C_t;\omega)$, we can conclude that $\sigma(C_t;\omega)=\sigma_{pt}(C_t;\omega)=\Lambda$. 
	The containment $\sigma(C_t;\omega)\subseteq \sigma^*(C_t;\omega)$ and the fact that $\sigma^*(C_t;\omega)$  is a closed set imply that $0\in \sigma^*(C_t;\omega)$. 
	
	It remains to show that every $\nu\not\in (\Lambda\cup\{0\})$ belongs to $\rho^*(C_t;\omega)$. So, fix $\nu\not\in (\Lambda\cup\{0\})$. Select $\delta>0$ such that the distance $\epsilon$ of $B(\nu,\delta)$ to the compact set $\Lambda\cup\{0\}$ is strictly positive. It follows from $0\leq t<1$ and the identity \eqref{eq.Inverse} which is coordinate $n$ of $(C_t-\nu I)^{-1}y$, for each $y\in\omega$, that for any given $k\in\N_0$ there exists $M_k>0$ such that 
	\[
	r_k((C_t-\mu I)^{-1}y)\leq \frac{M_k}{\epsilon^{k+1}}(\max_{0\leq j< k}|\nu|^j)r_{k}(y),\quad \mu\in B(\nu,\delta),
	\]
	where $r_k$ is the seminorm \eqref{sem-omega}, with $k$ in place of $n$. This implies that $\{(C_t-\mu I)^{-1}\, :\, \mu\in B(\nu,\delta)\}$ is a bounded set in $\cL_s(\omega)$ and hence, by the barrelledness of $\omega$, it is an equicontinuous subset of $\cL(\omega)$. Accordingly, $\nu\in \rho^*(C_t;\omega)$.
\end{proof}

\section{$C_t$ acting in the Fr\'echet spaces $\ell(p+)$, $d(p+)$ and $ces(p+)$}

Given $1\leq p<\infty$, consider any strictly decreasing sequence $\{p_k\}_{k\in\N}\subseteq (p,\infty)$ which satisfies $p_k\downarrow p$. Then $X_k:=\ell^{p_k}$ satisfies $X_{k+1}\subseteq X_k$ with $\|x\|_{\ell^{p_k}}\leq \|x\|_{\ell^{p_{k+1}}}$   for each $k\in\N$ and $x\in X_{k+1}$. 
 Moreover, $X=\cap_{k=1}^\infty X_k$ (i.e., $\ell(p+):=\cap_{k=1}^\infty \ell^{p_k}$) is a Fr\'echet space of the type given in Lemma \ref{Sp-FF} whose topology is generated by the increasing sequence of \textit{norms} $u_k$, for $k\in\N$, given by
\begin{equation}\label{eq.normelp}
	u_k\colon x\mapsto \|x\|_{\ell^{p_k}}, \quad x\in \ell(p+).
	\end{equation}
That is, $u_k\leq u_{k+1}$ for $k\in\N$. Moreover, $p_k>p$ implies that the natural inclusion map $\ell(p+)\hookrightarrow \ell^{p_k}$ is continuous for each $k\in\N$. Clearly the Banach space $\ell^p\subseteq \ell(p+)$ continuously and also $\ell(p+)\subseteq \omega$ continuously, as $\ell^q\subseteq \omega$ continuously, for every $1\leq q<\infty$ (cf. Lemma \ref{L-3-1}(iv)). The space $\ell(p+)$ is independent of the choice of $\{p_k\}_{k\in\N}$.

Changing the Banach spaces, now let $X_k:=ces(p_k)$, in which case again $X_{k+1}\subseteq X_k$ with $\|x\|_{ces(p_k)}\leq \|x\|_{ces(p_{k+1})}$   for each $k\in\N$ and $x\in X_{k+1}$;
 see \cite[Propostion 3.2(iii)]{ABR6}. Then $X=\cap_{k=1}^\infty X_k$ (i.e., $ces(p+):=\cap_{k=1}^\infty ces(p_k)$) is a Fr\'echet space of the type given in Lemma \ref{Sp-FF} whose topology is generated by the increasing sequence of \textit{norms} $v_k$, for $k\in\N$, given by
\begin{equation}\label{eq.norme-cesp}
	v_k\colon x\mapsto \|x\|_{ces(p_k)}, \quad x\in ces(p+).
\end{equation}
That is, $v_k\leq v_{k+1}$ for $k\in\N$. Again $ces(p)\subseteq ces(p+)$ (if $p>1$) and $ces(p+)\subseteq \omega$ with both inclusions continuous, where we again use Lemma \ref{L-3-1}(iv). The Fr\'echet spaces $ces(p+)$, for $1\leq p<\infty$, have been intensively studied in \cite{ABR7}, \cite{ABR8}.

Finally, consider the family of Banach spaces $X_k:=d_{p_k}$, in which case $X_{k+1}\subseteq X_k$ with $\|x\|_{d_{p_k}}\leq \|x\|_{d_{p_{k+1}}}$   for each $k\in\N$ and $x\in X_{k+1}$;
	  see \cite[Proposition 5.1(iii)]{BR1}. So,  $X=\cap_{k=1}^\infty X_k$ (i.e., $d(p+):=\cap_{k=1}^\infty d_{p_k}$) is a Fr\'echet space of the type given in Lemma \ref{Sp-FF} whose topology is generated by the increasing sequence of \textit{norms} $w_k$, for $k\in\N$, given by
\begin{equation}\label{eq.norme-dp}
	w_k\colon x\mapsto \|x\|_{d_{p_k}}, \quad x\in d(p+).
\end{equation}
That is, $w_k\leq w_{k+1}$ for $k\in\N$. With continuous inclusions we have $d_p\subseteq d(p+)\subseteq \omega$; see \cite[\S 4]{BR2} or, argue as for $\ell^p$ and $\ell(p+)$.

It is known that the canonical vectors $\mathcal{E}$ belong to $\ell(p+)$, $d(p+)$ and $ces(p+)$, for $1\leq p<\infty$, and form an \textit{unconditional basis} in each of these spaces; see \cite[Proposition 3.1]{BR2}, \cite[Lemma 4.1]{BR2} and \cite[Proposition 3.5(i)]{ABR7}, respectively.

In this section we consider the compactness and determine the spectra of $C_t$ when they act in the Fr\'echet spaces  $\ell(p+)$, $d(p+)$ and $ces(p+)$, for $1\leq p<\infty$. The decreasing sequence $\{p_k\}_{k\in\N}$ always has the properties listed above. Crucial for the proofs is the existence of a particular factorization available for $C_t$ (cf. Proposition \ref{P-4-4}).

The decreasing sequence $\varphi$ given in \eqref{dia} satisfies $\|\varphi\|_\infty=1$. Define the linear map $D_\varphi\colon \omega\to\omega$ by
\begin{equation}\label{Op-dia}
	D_\varphi x:=(\varphi_0x_0,\varphi_1 x_1,\varphi_2 x_2,\ldots)=\left(\frac{x_n}{n+1}\right)_{n\in\N_0},\quad x\in\omega.
	\end{equation}
The diagonal (multiplication) operator $D_\varphi\in \cL(\omega)$ since, for each $n\in\N_0$,
\[
r_n(D_\varphi x)\leq r_n(x),\quad x\in\omega,
\]
where $r_n$ is the seminorm \eqref{sem-omega}. Define the \textit{right-shift operator} $S\colon\omega\to\omega$ by
\begin{equation}\label{Opshift}
	Sx:=(0, x_0,x_1,\ldots),\quad x\in\omega.
	\end{equation}
For each $n\in\N$ note that $r_n(Sx)=\max_{0\leq k<n}|x_k|\leq r_n(x)$ and for $n=0$ that $r_0(Sx)=0\leq r_0(x)$ for each $x\in\omega$. So, for every $n\in\N_0$, the operator $S$ satisfies
\begin{equation}\label{est-S}
	r_n(Sx)\leq r_n(x),\quad x\in\omega,
	\end{equation}
which implies that $S\in \cL(\omega)$. The following result is Lemma 2.2 in \cite{CR4}.

\begin{lemma}\label{L-4-1} For each $t\in [0,1)$ we have the representation
	\[
	C_t=\sum_{n=0}^\infty t^nD_\varphi S^n
	\]
	with the series being convergent in $\cL_s(\omega)$. Equivalently, 
	\[
	C_tx=\sum_{n=0}^\infty t^nD_\varphi S^n x,\quad x\in\omega,
	\]
	with the series being convergent in $\omega$.
	\end{lemma}

Fix $t\in [0,1)$ and $x\in\omega$. For each $n\in\N_0$ it follows from \eqref{est-S} that
\[
r_n(\sum_{k=0}^\infty t^kS^kx)\leq \sum_{k=0}^\infty r_n(t^kS^kx)\leq \frac{1}{1-t}r_n(x).
\]
Accordingly, the series 
\begin{equation}\label{serie-R}
	R_t:=\sum_{n=0}^\infty t^nS^n, \quad t\in [0,1),
	\end{equation}
is absolutely convergent in the quasicomplete lcHs $\cL_s(\omega)$. In particular, $R_t\in \cL(\omega)$. Combining this with Lemma \ref{L-4-1} and the fact that $D_\varphi\in\cL(\omega)$ yields the following factorization of $C_t$.

\begin{prop}\label{P-4-2} For each $t\in [0,1)$ the operators $D_\varphi$, $R_t$, $C_t$ belong to $\cL(\omega)$ and 
	\begin{equation}\label{Rep}
C_t=D_\varphi R_t=\sum_{n=0}^\infty t^n D_\varphi S^n, 	
\end{equation}
with the series being absolutely convergent in $\cL_s(\omega)$.
\end{prop}

Our aim is to to extend Proposition \ref{P-4-2} to $\cL(X)$ with $X\in\{\ell(p+),ces(p+),d(p+)\,:\, 1\leq p<\infty\}$, to show that $D_\varphi\in\cL(X)$ is compact and then to apply Lemma \ref{Comp-Op} to conclude that $C_t\in\cL(X)$ is compact. 

\begin{prop}\label{P-4-3} Let $X$ be any Fr\'echet space in $\{\ell(p+),ces(p+),d(p+)\,:\, 1\leq p<\infty\}$. Then $D_\varphi$ maps $X$ into $X$ and $D_\varphi\in\cL(X)$ is compact.
\end{prop}

\begin{proof}
	Recall that $\varphi\in c_0$ with $\|\varphi\|_\infty=1$. We consider each of the three possible cases for $X$. It was shown above that $D_\varphi\in \cL(\omega)$ and that $X\subseteq \omega$ continuously.
	
	(a) Suppose that $X=\ell(p+)$ for some $1\leq p<\infty$. Clearly, $D_\varphi(X_k)\subseteq X_k$ for each $k\in\N$ and so $D_\varphi\in\cL(X)$; see Lemma \ref{C-incl}. In the notation of \cite{ABR8} it is clear from \eqref{Op-dia} that $D_\varphi$ is precisely the multiplication operator $M_\varphi$ defined there. Such a multiplication operator is compact if and only if $\varphi\in \ell(\infty-)=\cup_{s>1}\ell^s$, \cite[Proposition 17]{ABR8}, which is surely the case as $\varphi\in \ell^2$, for example. So, $D_\varphi\in \cL(\ell(p+))$ is a compact operator.
	
	(b) Suppose that $X=ces(p+)$ for some $1\leq p<\infty$. It follows from \eqref{eq.norm-ces} that $D_\varphi(X_k)\subseteq X_k$ for each $k\in\N$ and so $D_\varphi\colon X\to X$. Lemma \ref{C-incl} yields that $M_\varphi=D_\varphi\in\cL(ces(p+))$.
	 Moreover, if 
	  $\varphi\in d(\infty-)=\cup_{s>1} d_s$, then  $M_\varphi$ is also compact, \cite[Proposition 10]{ABR8}. But, $\varphi$ is a positive decreasing sequence and so $\varphi=\hat{\varphi}$. Accordingly, by choosing $s=2$ say, we see that
	  \[
	  \|\varphi\|_{d_2}:=\|\hat{\varphi}\|_2=\|\varphi\|_2<\infty.
	  \]
	  Hence, $\varphi\in d_2\subseteq d(\infty-)$ and so $D_\varphi=M_\varphi\in \cL(ces(p+))$ is indeed compact.
	  
	  (c) Suppose $X=d(p+)$ for some $1\leq p<\infty$. Since $|D_\varphi x|= D_\varphi|x|\leq |x|$, for $x\in\ell^\infty$, it is clear that $\widehat{D_\varphi x}\leq \hat{x}$. Then \eqref{norma} implies that $D_\varphi(X_k)\subseteq X_k$ for all $k\in\N$ and so $D_\varphi\colon X\to X$. Again Lemma \ref{C-incl} yields that $D_\varphi\in \cL(d(p+))$. Note that the operator
	   $M_{d(p+)}^\varphi$ in \cite{BR3} is precisely $D_\varphi\colon d(p+)\to d(p+)$. It was verified in (b) above that $\varphi\in d(\infty-)$ which, together with $D_\varphi\in \cL(d(p+))$, implies that $D_\varphi$ is compact, \cite[Theorem 4.13(i)]{BR3}.
\end{proof}

\begin{prop}\label{P-4-4} Let $t\in [0,1)$, and $X$ be any  Fr\'echet space in $\{\ell(p+),ces(p+),d(p+)\,:\, 1\leq p<\infty\}$.
	
	\begin{itemize}
		\item[\rm (i)] The generalized Cesàro operator $C_t$ maps $X$ into itself and $C_t\in \cL(X)$.
		\item[\rm (ii)] The right-shift operator $S$ given by \eqref{Opshift} maps $X$ into itself and belongs to $ \cL(X)$.
		\item[\rm (iii)] The operator $R_t$ given by \eqref{serie-R} maps $X$ into itself and belongs to $ \cL(X)$, with the series $\sum_{n=0}^\infty t^nS^n$ being absolutely convergent in $\cL_s(X)$. Moreover,
		\[
		C_t=D_\varphi R_t=\sum_{n=0}^\infty t^nD_\varphi S^n.
		\]
	\end{itemize}
\end{prop}

\begin{proof}
	(i) Again we consider  the three possible cases for $X$. Fix $t\in [0,1)$. According to Proposition \ref{nocompact-omega} the operator $C_t\in \cL(\omega)$.
	
	(a) Suppose that $X=\ell(p+)$ for some $1\leq p<\infty$. Proposition \ref{Sp-spazilp} implies that $C_t(X_k)\subseteq X_k$ for all $k\in\N$, with $X_k=\ell^{p_k}$, and so $C_t(X)\subseteq X$. In view of  Lemma \ref{C-incl}, with $T:=C_t$, it follows that $C_t\in \cL(\ell(p+))$.
	
	(b) Suppose that $X=ces(p+)$ for some $1\leq p<\infty$. Proposition \ref{Sp-spazices} shows that $C_t(X_k)\subseteq X_k$ for all $k\in\N$, with $X_k=ces(p_k)$, and so $C_t(X)\subseteq X$. Again, for $T:=C_t$,  Lemma \ref{C-incl} implies that $C_t\in \cL(ces(p+))$.
	
	(c) Suppose that $X=d(p+)$ for some $1\leq p<\infty$. Proposition \ref{Sp-spazidp} shows that $C_t(X_k)\subseteq X_k$ for all $k\in\N$, with $X_k=d_{p_k}$, and so $C_t(X)\subseteq X$. Yet again, for $T:=C_t$,  Lemma \ref{C-incl} implies that $C_t\in \cL(d(p+))$.
	
	(ii) Again we check the three separate cases for $X$. Prior to Lemma \ref{L-4-1} it was shown that $S\in \cL(\omega)$.
	
	(a) Suppose that $X=\ell(p+)$ for some $1\leq p<\infty$. Using the fact that the Banach space right-shift operator $S\colon \ell^{p_k}\to\ell^{p_k}$ is an isometry, for every $k\in\N$, we see that $S(X)\subseteq X$.
 It follows that $S\in \cL(\ell(p+))$; see Lemma \ref{C-incl} for $T:=S\in \cL(\omega)$.
	
	(b) Suppose that $X=ces(p+)$ for some $1\leq p<\infty$. It is known, for each $k\in\N$, that $S\in \cL(ces(p_k))$ and $\|S\|_{ces(p_k)\to  ces(p_k)}\leq 1$, \cite[Lemma 5.4]{CR4}. Accordingly, $S(X)\subseteq X$ and so  Lemma \ref{C-incl}, for $T:=S\in \cL(\omega)$,
 implies that $S\in \cL(ces(p+))$.
	
	(c) Suppose that $X=d(p+)$ for some $1\leq p<\infty$. Fix $k\in\N$. It is known that 
	$S\in \cL(d_{p_k})$ and 
	\begin{equation}\label{eq.4-11}
		\|S^m\|_{d_{p_k}\to d_{p_k}}=(m+1)^{1/p_k},\quad m\in\N_0,
		\end{equation}
	\cite[Lemma 6.2]{CR4}. For $m=1$ we can conclude that $S(d_{p_k})\subseteq d_{p_k}$ for $k\in\N$, that is, $S(X)\subseteq X$. So, in view of Lemma \ref{C-incl}, for $T:=S\in\cL(\omega)$, it follows that $S\in \cL(d(p+))$.
	
	(iii) (a) Suppose that $X=\ell(p+)$ for some $1\leq p<\infty$. Fix $k\in\N$ and $x\in \ell(p+)\subseteq \ell^{p_k}$. It follows from $S$ being an isometry in $\ell^{p_k}$ that $u_k(S^n x)= u_k(x)$ for all $n\in\N_0$ and hence, that
	\[
	\sum_{n=0}^\infty u_k(t^n S^n x)=\sum_{n=0}^\infty t^n u_k(S^n x)\leq \frac{1}{1-t}u_k(x)<\infty.
	\]
	Accordingly, the series $\sum_{n=0}^\infty t^n S^n x$ is absolutely convergent in the Fr\'echet space $\ell(p+)$ for each $x\in \ell(p+)$. By part (ii) the sequence $\{\sum_{n=0}^m t^n S^n \}_{m\in\N_0}\su\cL(\ell(p+))$ and so, by the Banach-Steinhaus theorem (as $\ell(p+)$ is barrelled), the series $\sum_{n=0}^\infty t^n S^n$ is absolutely convergent in $\cL_s(\ell(p+))$; its sum is denoted by $R_t\in \cL(\ell(p+))$.
	
	It has been established that each of the operators $C_t$, $D_\varphi$, $R_t$ belongs to $\cL(\ell(p+))$. The identities
	$C_t=D_\varphi R_t=\sum_{n=0}^\infty t^nD_\varphi S^n$
	are valid in $\cL(\ell(p+))$ because they are valid in $\cL(\omega)$; see Lemma \ref{L-4-1} and both \eqref{serie-R} and \eqref{Rep}.
	
	(b) Suppose $X=ces(p+)$ for some $1\leq p<\infty$. Fix $k\in\N$ and $x\in ces(p+)\subseteq ces(p_k)$. Using $\|S^n\|_{ces(p_k)\to ces(p_k)}\leq 1$, for all $n\in\N_0$ (see the proof of part (ii)(b)), we can argue as in (a) to conclude that
	\[
	\sum_{n=0}^\infty v_k(t^nS^n x)\leq \frac{1}{1-t}v_k(x)<\infty.
	\]
	Hence, the series $\sum_{n=0}^\infty t^n S^n x$ is absolutely convergent in $ces(p+)$ for each $x\in ces(p+)$. Then argue as in (a) to deduce that the series $R_t:=\sum_{n=0}^\infty t^n S^n $ is absolutely convergent in $\cL_s(ces(p+))$, with $R_t\in \cL(ces(p+))$, and that the identities
	$C_t=D_\varphi R_t=\sum_{n=0}^\infty t^nD_\varphi S^n$ 
	are valid in $\cL(ces(p+))$.
	
	(c) Let $X=d(p+)$ for some $1\leq p<\infty$. Fix $k\in\N$ and $x\in d(p+)\subseteq d_{p_k}$. It follows from \eqref{eq.4-11} that 
	\[
	w_k(S^mx)=\|S^m x\|_{d_{p_k}}\leq \|S^m\|_{d_{p_k}\to d_{p_k}} \|x\|_{d_{p_k}}=(m+1)^{1/p_k}w_k(x),\quad m\in\N_0,
	\]
	and hence, since $0\leq t<1$, that
	\[
	\sum_{n=0}^\infty w_k(t^nS^n x)\leq (\sum_{n=0}^\infty t^n (n+1)^{1/p_k})w_k(x)<\infty.
	\]
	Now argue as in (a) to conclude that the series $R_t:=\sum_{n=0}^\infty t^n S^n$ is absolutely convergent in $\cL_s(d(p+))$, with $R_t\in \cL(d(p+))$, and that the identities
$C_t=D_\varphi R_t=\sum_{n=0}^\infty t^nD_\varphi S^n$ are valid in $\cL(d(p+))$. 
\end{proof}

We come to the main result of this section, which should be compared with Proposition \ref{nocompact-omega} and Theorem \ref{Sp-omega}.

\begin{theorem}\label{Sp-F} Let $t\in [0,1)$ and $X$ be any Fr\'echet space in $\{\ell(p+), ces(p+),d(p+):\, 1\leq p<\infty\}$.
	\begin{itemize}
		\item[\rm (i)] The generalized Cesàro operator $C_t\in \cL(X)$ is compact.
		\item[\rm (ii)] The spectra of $C_t$ are given by
		\begin{equation}\label{spectra-pt-X}
			\sigma_{pt}(C_t;X)=\Lambda
		\end{equation}
	and
	\begin{equation}\label{spectra-X}
		\sigma^*(C_t;X)=\sigma(C_t;X)=\Lambda\cup\{0\}.
	\end{equation}
\item[\rm (iii)] For each $\lambda\in \sigma_{pt}(C_t;X)$ the subspace $(\lambda I-C_t)(X)$ is closed in $X$ with ${\rm codim}\, (\lambda I- C_t)(X)=1$. Moreover, the $1$-dimensional eigenspace $\Ker(\frac{1}{m+1}I-C_t)={\rm span} (x^{[m]})$, for each $m\in\N_0$, with $x^{[m]}\in d_1\subseteq X$ given by \eqref{eigenvalue}.
	\end{itemize}
\end{theorem}

\begin{proof}
	(i) Since $D_\varphi\in \cL(X)$ is compact (cf. Proposition \ref{P-4-3}) and $R_t\in \cL(X)$ (cf. Proposition \ref{P-4-4}(iii)), the compactness of $C_t$ follows from the factorization $C_t=D_\varphi R_t$ (cf. Proposition \ref{P-4-4}(iii)) and Lemma \ref{Comp-Op}.
	
	(ii) Since $X\subseteq \omega$, we can conclude from Theorem \ref{Sp-omega} that
	\begin{equation}\label{eq.4-14}
		\sigma_{pt}(C_t;X)\subseteq \sigma_{pt}(C_t;\omega)=\Lambda.
		\end{equation}
	Fix $1\leq p<\infty$. Then $d_1\subseteq \ell^1\subseteq \ell^p\subseteq \ell(p+)$. Since $\ell^p\subseteq ces(p)\subseteq ces(p+)$ (cf. (1) on p.2 of \cite{CR3}), it follows that also $d_1\subseteq ces(p+)$. Moreover, $d_1\subseteq d_p\subseteq d(p+)$. So, $d_1\subseteq X$. Given $\nu\in\Lambda$ there exists $m\in\N_0$ such that $\nu=\varphi_m$. According to Lemma \ref{L-3-3} the $1$-dimensional eigenspace corresponding to $\nu\in \sigma_{pt}(C_t;\omega)$ is spanned by $x^{[m]}$ with $x^{[m]}\in d_1$. Since $d_1\subseteq X$, it follows that $\nu\in \sigma_{pt}(C_t;X)$. So, it has been established that $\Lambda\subseteq \sigma_{pt}(C_t;X)$. Combined with \eqref{eq.4-14} we can conclude that \eqref{spectra-pt-X} is valid.
	
	The spectrum of a compact operator in a lcHs is necessarily a compact subset of $\C$ (see \cite[Theorem 9.10.2]{Ed}, \cite[Theorem 4 \& Proposition 6]{Gr}) and it is either a finite set or a countable sequence of non-zero eigenvalues with limit point $0$. It follows from part (i) and \eqref{spectra-pt-X} that 
	\begin{equation}\label{eqq.4-15}
		\sigma(C_t;X)=\Lambda\cup\{0\}.
	\end{equation}
The discussion in the first three paragraphs of this section, with the notation from there, shows that $X=\cap_{k=1}^\infty X_k$ is a Fr\'echet space of the type given in Lemma \ref{Sp-FF}. Setting there $T:=C_t\in \cL(X)$ and $T_n:=C_t\in \cL(X_n)$ for $n\in\N$ (see Propositions \ref{Sp-spazilp}, \ref{Sp-spazices} and \ref{Sp-spazidp}),
it is  clear that condition (A) is satisfied. Moreover, $\sigma(T_n;X_n)=\Lambda\cup\{0\}$ for every $n\in\N$ (cf. \eqref{Sp-lp}, \eqref{Sp-ces} and \eqref{Sp-dp} with $p_n$ in place of $p$) and so, via \eqref{eqq.4-15}, we have that
\[
\cup_{n=1}^\infty\sigma(T_n;X_n)=\Lambda\cup\{0\}=\sigma(T;X)=\sigma(C_t;X).
\]
In particular, $\cup_{n=1}^\infty\sigma(T_n;X_n)\subseteq \ov{\sigma(T;X)}$ and so we can conclude from Lemma \ref{Sp-FF} that \eqref{spectra-X} is valid.

(iii) First observe that $(\nu I-C_t)=\nu (I- \nu^{-1}C_t)$, for $\nu\in\C\setminus \{0\}$, with $\nu^{-1}C_t$ being a compact operator by part (i). So, by \cite[Theorem 9.10.1(i)]{Ed}, the subspace $(\nu I-C_t)(X)$ is closed in $X$ with ${\rm codim}\, (\nu I-C_t)(X)={\rm dim}\,\Ker (\nu I-C_t)$ for every $\nu\in \sigma_{pt}(C_t;X)$. But, ${\rm dim}\,\Ker (\nu I-C_t)=1$ for  $\nu\in \sigma_{pt}(C_t;X)$, as observed in the proof of part (ii), where it was also established that $\Ker(\frac{1}{m+1}I -C_t)={\rm span}(x^{[m]})$, for each $m\in\N_0$.
%
%
%
\end{proof} 

\begin{remark}\label{R-4-6}
	\rm (i) The identity \eqref{spectra-pt-X}, established in the proof of part (ii) of Theorem \ref{Sp-F}, can also be deduced from Lemma \ref{Sp-FF}(ii).

	(ii) Let $t\in [0,1)$ and $X$ be any Fr\'echet space in $\{\ell(p+), ces(p+), d(p+)\,:\, 1\leq p<\infty\}$. Since $X\subseteq \omega$ and $C_t\in \cL(\omega)$ is injective (cf. Lemma \ref{L-3-4}), also $C_t\in \cL(X)$ is injective. Moreover, as $C_t\in \cL(X)$ is compact (cf. Theorem \ref{Sp-F}(i)) it cannot be surjective, otherwise it would be an isomorphism thereby implying that $0\in \rho(C_t;X)$, which is \textit{not} the case (see \eqref{spectra-X}). Recall that $\mathcal{E}$ is a basis for $X$ and, by Lemma \ref{L-3-1}(iii), that the range $C_t(X)$ is a proper, dense subspace of $X$. Hence, $0$ belongs to the \textit{continuous spectrum} of $C_t$. This is in contrast to the situation of $\omega$, where $0\in\rho(C_t;\omega)$; see Theorem \ref{Sp-omega}.
	
	(iii) Concerning the case when $t=1$, it is known that $\sigma_{pt}(C_1;\ell(p+))=\emptyset$ and
	\begin{equation}\label{eq.4-16}
		\sigma(C_1;\ell(p+))=\left\{z\in\C\, :\, \left|z-\frac{p'}{2}\right|<\frac{p'}{2}\right\}\cup\{0\}\ \mbox{ and }\ \sigma^*(C_1;\ell(p+))=\overline{\sigma(C_1;\ell(p+))},
	\end{equation}
for every $1<p<\infty$, \cite[Theorem 2.2]{ABR3}. For $p=1$, again $\sigma_{pt}(C_1;\ell(1+))=\emptyset$ whereas
\begin{equation}\label{eq.4-17}
	\sigma(C_1;\ell(1+))=\left\{z\in\C\, :\, {\rm Re}\,z>0\right\}\cup\{0\}\ \mbox{ and }\ \sigma^*(C_1;\ell(1+))=\ov{\sigma(C_1;\ell(1+))},
\end{equation}
\cite[Theorem 2.4]{ABR3}. For the Fr\'echet space $ces(p+)$, both \eqref{eq.4-16} and \eqref{eq.4-17} are also valid (with $ces(p+)$, resp. with $ces(1+)$, in place of $\ell(p+)$, resp. in place of $\ell(1+)$), as well as $\sigma_{pt}(C_1;ces(p+))=\emptyset$ for all $1\leq p<\infty$, \cite[Theorem 3]{ABR8}. For the Fr\'echet space $d(p+)$,  both \eqref{eq.4-16} and \eqref{eq.4-17} are again valid with $d(p+)$ (resp. with $d(1+)$), in place of $\ell(p+)$ (resp. of $\ell(1+)$), as well as $\sigma_{pt}(C_1;d(p+))=\emptyset$ for all $1\leq p<\infty$, \cite[Theorem 3.2]{BR3}.
\end{remark}

\section{$C_t$ acting in the (LB)-spaces $\ell(p-)$, $d(p-)$ and $ces(p-)$}

Given $1<p\leq \infty$, consider any strictly increasing sequence $\{p_k\}_{k\in\N}\su (1,p)$ which satisfies $p_k\uparrow p$. The Banach spaces $X_k:=\ell^{p_k}$ satisfy $X_k\subset X_{k+1}$ with a continuous inclusion, for each $k\in\N$, and $X=\cup_{k=1}^\infty X_k$ is an (LB)-space, necessarily \textit{regular} by Lemma \ref{Reg-LB}. The (LB)-space $X$ is denoted by $\ell(p-)=\ind_k\ell^{p_k}$. If we set $X_k:=ces(p_k)$, then again $X_k\subset X_{k+1}$ for $k\in\N$ (see the discussion prior to Proposition 3.3 in \cite{ABR6}) with a continuous inclusion. The (LB)-space $X:=\cup_{k=1}^\infty X_k$, necessarily \textit{regular} by Lemma \ref{Reg-LB}, is denoted by $ces(p-):=\ind_k ces(p_k)$. Finally, the Banach spaces $X_k:=d_{p_k}$ satisfy $X_k\subset X_{k+1}$ with a continuous inclusion, for $k\in\N$ (see Propositions 2.7(ii) and 5.1(iii) in \cite{BR1}). The (LB)-space $X:=\cup_{k=1}^\infty X_k$, necessarily \textit{regular} by Lemma \ref{Reg-LB}, is denoted by $d(p-):=\ind_k d_{p_k}$. The discussion after \eqref{d1} shows that $d_1$ is continuously included in each space in $\{\ell^p, ces(p), d_p\,:\,1<p<\infty\}$, from which it follows that $d_1\subseteq X$ continuously, for each $X\in\{\ell(p-), ces(p-),d(p-)\,:\, 1<p\leq \infty\}$. Indeed, by the definition of the inductive limit topology, $\ell^p\subseteq \ell(p-)$ and $d_p\subseteq d(p-)$ and $ces(p)\subseteq ces(p-)$ with all inclusions continuous.
In all of these (LB)-spaces the canonical vectors $\mathcal{E}$ form a Schauder basis. Indeed, concerning $\ell(p-)$ recall that $\cE$ is a basis for each Banach space $\ell^{p_k}$ and the natural inclusion $\ell^{p_k}\subseteq \ell(p-)$ is continuous for each $k\in\N$. It follows that $\cE$ is a Schauder basis for $\ell(p-)$. For the (LB)-spaces $ces(p-)$, resp. $d(p-)$, see \cite[Proposition 2.1]{ABR4}, resp. \cite[Theorem 4.6]{BR2}. It follows from \cite[Proposition 24.7]{MV} together with Lemma \ref{L-3-1}(iv) that $X\subseteq \omega$ continuously. For further properties of the (LB)-spaces $\ell(p-)$, $ces(p-)$ and $d(p-)$, and operators acting in them, we refer to \cite{ABR4}, \cite{BR2}, \cite{BR3}, for example, and the references therein.

For each of the three cases above it is clear that the diagonal (multiplication) operator $D_\varphi\in\cL(\omega)$ as defined in \eqref{Op-dia} satisfies $D_\varphi(X_k)\subseteq X_k$ for all $k\in\N$ (cf. proof of Proposition \ref{P-4-3}) and so $D_\varphi(X)\subseteq X$.  By Lemma \ref{C-incl} it follows that $D_\varphi\in \cL(X)$. Actually, $D_\varphi\in \cL(X)$ is a compact operator. For the case $X=\ell(p-)$, since $\varphi\in \ell^2\subseteq \ell(\infty-)$, Proposition 4.5 of \cite{ABR4} implies that $D_\varphi\in \cL(\ell(p-))$ is compact. Suppose now that $X:=ces(p-)$. By Proposition 4.2 of \cite{ABR4} it follows that $D_\varphi\in \cL(ces(p-))$ is compact provided that $\hat{\varphi}\in \ell^t$ for some $t>q$ (with $\frac{1}{p}+\frac{1}{q}=1$). But, it is clear from \eqref{dia} that $\hat{\varphi}=\varphi\in \cap_{s>1}\ell^s$ and so $D_\varphi$ is a compact operator in $ces(p-)$. Consider now when $X:=d(p-)$. Since $\hat{\varphi}\in\ell^2$ and $\hat{\varphi}=\varphi$, it follows that $\varphi\in d_2\subseteq d(\infty-)$ and so Proposition 4.13(ii) of \cite{BR3} implies that $D_\varphi$ is a compact operator in $d(p-)$. So, we have established the following result.

\begin{prop}\label{P-5-1} Let $X$ be any (LB)-space in $\{\ell(p-), ces(p-), d(p-)\, :\, 1<p\leq \infty\}$. Then $D_\varphi$ maps $X$ into itself and $D_\varphi\in \cL(X)$ is a compact operator.
\end{prop}

The following result will also be required.

\begin{prop}\label{P-5-2} Let $t\in [0,1)$ and $X$ be any (LB)-space in $\{\ell(p-), ces(p-), d(p-)\, :\, 1<p\leq \infty\}$. 
	\begin{itemize}
		\item[\rm (i)] The right-shift operator $S$ given by \eqref{Opshift} maps $X$ into $X$ and belongs to $\cL(X)$.
		\item[\rm (ii)] The generalized Cesàro operator $C_t$ maps $X$ into $X$ and satisfies $C_t\in \cL(X)$.
		\item[\rm (iii)] The operator $R_t$ given by \eqref{serie-R} maps $X$ into $X$ and belongs to $\cL(X)$, with the series $\sum_{n=0}^\infty t^nS^n$ being convergent in $\cL_s(X)$. Moreover,
		\begin{equation}\label{eq.5-1}
			C_t=D_\varphi R_t=\sum_{n=0}^\infty t^n D_\varphi S^n.
			\end{equation}
	\end{itemize}
	\end{prop}

\begin{proof}
	(i) It was observed in the proof of Proposition \ref{P-4-4}(ii) that $S\in \cL(\omega)$ as well as $S(\ell^{p_k})\subseteq \ell^{p_k}$ and  $S(ces(p_k))\subseteq ces(p_k)$ and  $S(d_{p_k})\subseteq d_{p_k}$, for each $k\in\N$, from which it is clear that $S(X)\subseteq X$. By Lemma \ref{C-incl} it follows that $S\in \cL(X)$.
	
	(ii) In each of the three cases $\ell(p-)$, $ces(p-)$, $d(p-)$ for $X$ it is clear that $C_t\colon\omega\to\omega$ (cf. \eqref{Ces-op}) satisfies $C_t(X_k)\subseteq X_k$ for all $k\in\N$ (see the proof of Proposition \ref{P-4-4}(i)) and hence, $C_t(X)\subseteq X$. Since $C_t\in \cL(\omega)$, via Proposition \ref{nocompact-omega}, again by Lemma \ref{C-incl} we can conclude that $C_t\in  \cL(X)$.
	
	(iii) According to part (i) the sequence $\{\sum_{n=0}^kt^nS^n\}_{k\in\N_0}\subseteq \cL(X)$.
	
	\textit{Claim}. $\{\sum_{n=0}^kt^nS^n\, :\, k\in\N_0\}$ \textit{is an equicontinuous subset of } $\cL(X)$.
	
	Suppose first that $X=\ell(p-)$ or $X=ces(p-)$. Since $X$ is barrelled, to establish the \textit{Claim} it suffices to show, for each $x\in X$, that
\[
B(x):=\{\sum_{n=0}^kt^nS^nx\, :\, k\in\N_0\}
\]
is a bounded subset of $X=\ind_r X_r$. Since $X$ is a regular (LB)-space, the set $B(x)$ will be bounded if there exists $m\in\N$ such that $B(x)\subseteq X_m$ and $B(x)$ is bounded in the Banach space $X_m$. But, $x\in X=\cup_{r=1}^\infty X_r$ and so there exists $m\in\N$ such that $x\in X_m$. Since $S^n\in \cL(X_m)$ for all $n\in\N_0$, it is clear that $B(x)\subseteq X_m$. Moreover, in the proof of Proposition \ref{P-4-4}(ii) it was noted that $\|S\|_{X_m\to X_m}\leq 1$ and hence, $\|S^n\|_{X_m\to X_m}\leq 1$ for all $n\in\N_0$. Accordingly,
\[
\|\sum_{n=0}^k t^n S^n x\|_{X_m}\leq \sum_{n=0}^\infty t^n \|S^n x\|_{X_m}\leq \sum_{n=0}^\infty t^n \|S^n\|_{X_m\to X_m}\|x\|_{X_m}\leq \frac{\|x\|_{X_m}}{(1-t)},\quad k\in\N_0,
\]
which implies that $B(x)$ is a bounded set in $X_m$. In the event that $X=d(p-)$, an analogous argument applies except that now $X_m=d_{p_m}$ and so $\|S^n\|_{d_{p_m}\to d_{p_m}}=(n+1)^{1/p_m}$ for $n\in\N_0$; see \eqref{eq.4-11}. In this case the previous inequality becomes
\[
\|\sum_{n=0}^k t^n S^n x\|_{d_{p_m}}\leq (\sum_{n=0}^\infty t^n (n+1)^{1/p_m})\|x\|_{d_{p_m}},\quad k\in\N_0, 
\]
which implies that $B(x)$ is a bounded set in $d_{p_m}$ as $\sum_{n=0}^\infty t^n (n+1)^{1/p_m}<\infty$. The proof of the \textit{Claim} is thereby complete.

In view of the \textit{Claim}, to show that the series $\sum_{n=0}^\infty t^n S^n$ converges in $\cL_s(X)$ it suffices to show that the limit
\begin{equation}\label{eq.5-2}
	R_tx:=\lim_{k\to\infty}\sum_{n=0}^k
t^n S^n x=\sum_{n=0}^\infty t^n S^n x
\end{equation}
exists in $X$ for all $x\in X$ in some dense subset of $X$. Since $\cE$ is a Schauder basis for $X$, its linear span ${\rm span}\, \cE$ is a dense subspace of $X$ and so it suffices to show that the limit in \eqref{eq.5-2} exists for each $x\in \cE$. Let 
$x:=e_r=(0,\ldots, 0,1,0, \ldots)$,
for any fixed $r\in\N_0$, where $1$ is in position $r$. Then $S^ne_r=e_{r+n}$ for all $n\in\N_0$. Fix $k\in\N_0$. It follows that 
\begin{equation}\label{eq.5-3}
	\sum_{n=0}^kt^nS^ne_r=\sum_{n=0}^kt^n e_{r+n}=(0,\ldots, 1, t, t^2,\ldots, t^k, 0,0,\ldots),
	\end{equation}
where $1$ is in position $r$ and $t^k$ is in position $r+k$. Observe that $\|e_j\|_{\ell^{p_1}}=1$ for $j\in\N_0$. Direct calculation via \eqref{norma} shows that $\|e_j\|_{d_{p_1}}=(j+1)^{1/p_1}$, for $j\in\N_0$, and by Lemma 4.7 in \cite{Be}, there exists $K>0$ such that $\|e_j\|_{ces(p_1)}\leq K$ for all $j\in\N_0$. It follows that $\sum_{j=r}^\infty t^j \|e_j\|_{\ell^{p_1}}=\frac{t^r}{(1-t)}\leq \frac{1}{(1-t)}$, that $\sum_{j=r}^\infty t^j \|e_j\|_{ces(p_1)}\leq \frac{Kt^r}{(1-t)}\leq \frac{K}{(1-t)}$ and that $\sum_{j=r}^\infty t^j \|e_j\|_{d_{p_1}}\leq \sum_{j=r}^\infty t^j(j+1)^{1/p_1}<\infty$. Accordingly, the series
	\begin{equation}\label{eq.5-4}
		y^{[r]}:=\sum_{j=r}^\infty t^je_j=(0,\ldots,0, 1,t, t^2, \ldots),
		\end{equation}
	with $1$ in position $r$, is absolutely convergent in the Banach space $X_1$ belonging to $\{\ell^{p_1}, ces(p_1), d_{p_1}\}$ and defines an element of $X_1$, that is, $y^{[r]}\in X_1$. Since the inclusion $X_1\subseteq X$ is continuous, the series \eqref{eq.5-4} is also convergent to $y^{[r]}$ in $X$. For any $k>r$ we have
	\[
	\|y^{[r]}-\sum_{n=0}^kt^n S^n e_r\|_{X_1}=\|\sum_{j=r+k+1}^\infty t^j e_j	\|_{X_1}\to 0, \quad k\to\infty,
	\]
	being the tail of the absolutely convergent series \eqref{eq.5-2}. So, the sequence in \eqref{eq.5-3} converges to $y^{[r]}$ in $X_1$ for $k\to\infty$ and hence, also to $y^{[r]}$ in $X$. Since $r\in\N_0$ is arbitrary, we have proved that the limit in \eqref{eq.5-2} exists in $X$ for each $x\in {\rm span}\, \cE$ and hence, by the \textit{Claim}, it exists for \textit{every} $x\in X$. Accordingly, the limit operator $R_t=\lim_{k\to\infty}\sum_{n=0}^k t^n S^n$ exists in $\cL_s(X)$. Since $D_\varphi, R_t, C_t\in \cL(X)$ and $X\subseteq \omega$ continuously, the equality $C_t=D_\varphi R_t=\sum_{n=0}^\infty t^n D_\varphi S^n$ follows from Proposition \ref{P-4-2}.
\end{proof}

The main result of this section is as follows.

\begin{theorem}\label{Spectrum-LB} Let $t\in [0,1)$ and $X$ be any (LB)-space in $\{\ell(p-), ces(p-),d(p-):\, 1<p\leq\infty\}$.
	\begin{itemize}
		\item[\rm (i)] The generalized Cesàro operator $C_t\in \cL(X)$ is compact.
		\item[\rm (ii)] The spectra of $C_t$ are given by
		\begin{equation}\label{spectra-pt-LB}
			\sigma_{pt}(C_t;X)=\Lambda
		\end{equation}
		and
		\begin{equation}\label{spectra-LB}
			\sigma^*(C_t;X)=\sigma(C_t;X)=\Lambda\cup\{0\}.
		\end{equation}
		\item[\rm (iii)] For each $\lambda\in \sigma_{pt}(C_t;X)$ the subspace $(\lambda I-C_t)(X)$ is closed in $X$ with ${\rm codim}\, (\lambda I- C_t)(X)=1$. Moreover, the $1$-dimensional eigenspace $\Ker (\frac{1}{m+1}I-C_t)={\rm span}(x^{[m]})$, for each $m\in\N_0$, with $x^{[m]}\in d_1\subseteq X$ given by \eqref{eigenvalue}.
	\end{itemize}
	\end{theorem}

\begin{proof}
	(i) Since $D_\varphi\in\cL(X)$ is compact (cf. Proposition \ref{P-5-1}) and $R_t\in \cL(X)$ (cf. Proposition \ref{P-5-2}(iii)), the compactness of $C_t\in \cL(X)$ follows from the factorization in \eqref{eq.5-1} and Lemma \ref{Comp-Op}.
	
	(ii) The (LB)-space $X=\ind_k X_k$ is an inductive limit of the type in Lemma \ref{Sp-LB}. Moreover, $T:=C_t\in \cL(X)$ has the property, for each $k\in\N$, that the restriction $T_k$ of $T$ to the Banach space $X_k$ maps $X_k$ into itself and satisfies $T_k\in \cL(X_k)$. That is, $T$ satisfies condition (A\'\,) of Lemma \ref{Sp-LB}. Then, by  Lemma \ref{Sp-LB}(i) it follows that $\sigma_{pt}(C_t;X)=\cup_{k=1}^\infty\sigma_{pt}(T_k; X_k)=\Lambda$ (cf. Propositions \ref{Sp-spazilp}, \ref{Sp-spazices} and \ref{Sp-spazidp}). Since $C_t\in \cL(X)$ is compact by part (i), the analogous argument used to prove \eqref{eqq.4-15}, now with \eqref{spectra-pt-X} replaced by \eqref{spectra-pt-LB}, can be used to show that 
	\begin{equation}\label{eq.5-8}
		\sigma(C_t;X)=\Lambda\cup\{0\}.
		\end{equation}
	Moreover, $\sigma(T_k;X_k)=\sigma (C_t;X_k)=\Lambda\cup \{0\}$ for every $k\in\N$ and so, for $m=1$ say, we note (via \eqref{eq.5-8}) that
	\[
	\cup_{k=m}^\infty\sigma(T_k; X_k)=\Lambda \cup\{0\}\subseteq \ov{\sigma(T;X)}.
	\]
	We can conclude again from Lemma \ref{Sp-LB}(ii) that $\sigma^*(C_t;X)=\ov{\sigma(C_t;X)}$. Combined with \eqref{eq.5-8} this yields \eqref{spectra-LB}.

	(iii) The analogous argument used to prove part (iii) of Theorem \ref{Sp-F} also applies to establish the given statement. Again, since $d_1\subseteq X$ (see the introduction to Section 5), it follows that  $\Ker (\frac{1}{m+1}I-C_t)={\rm span}(x^{[m]})$, for each $m\in\N_0$.
\end{proof}

\begin{remark}\label{R-5-4}
	\rm (i) An examination of the arguments given in Remark \ref{R-4-6} shows that, when suitably adapted, they also apply here to conclude that $C_t(X)$ is a proper, dense subspace of $X$. That is, $0$ belongs to the \textit{continuous spectrum} of $C_t$. 
	
	(ii) Concerning $t=1$, it is known that $\sigma_{pt}(C_1;ces(p-))=\emptyset$, \cite[Proposition 3.1]{ABR4} with 
	\begin{equation}\label{eq.5-9}
		\{0\}\cup\left\{z\in\C\,:\, \left|z-\frac{p'}{2}\right|<\frac{p'}{2}\right\}\subseteq \sigma(C_1; ces(p-))\subseteq \left\{z\in\C\,:\, \left|z-\frac{p'}{2}\right|\leq\frac{p'}{2}\right\}
		\end{equation}
	and
	\begin{equation}\label{eq.5-10}
		\sigma^*(C_1;ces(p-))=\left\{z\in \C\,:\, \left|z-\frac{p'}{2}\right|\leq\frac{p'}{2}\right\}=\ov{\sigma(C_1;ces(p-))},\ 1<p\leq\infty,
		\end{equation}
	\cite[Propositions 3.2 \& 3.3]{ABR4}.
	
	For the (LB)-space $d(p-)$, both \eqref{eq.5-9} and \eqref{eq.5-10} are also valid (with $d(p-)$ in place of $ces(p-)$) as well as $\sigma_{pt}(C_1; d(p-))=\emptyset$, for all $1<p\leq\infty$; see  Theorem 3.6 in \cite{BR3}.
\end{remark}

The spectrum of $C_1$ acting in $\ell(p-)$ is covered by the next result. 

 Recall that the space  $\ell(p'+)$ is the strong dual of $\ell(p-)$, \cite[Proposition 3.4(i)]{BR2}, and that the dual operator $C_1'\in \cL(\ell(p'+))$ of $C_1\in \cL(\ell(p-))$ is given by
\[
C_1'x=\left(\sum_{i=n}^\infty\frac{x_i}{i+1}\right)_{n\in\N_0},\quad x=(x_n)_{n\in\N_0}\in \ell(p'+),
\]
see, for instance, \cite[p.123]{Le}.

\begin{prop}\label{Sp-lp-} Let $p\in (1,\infty]$ and let $p'\in [1,\infty)$ satisfy $\frac{1}{p}+\frac{1}{p'}=1$.
	\begin{itemize}
		\item[\rm (i)] $\sigma_{pt}(C_1;\ell(p-))=\emptyset$   and $\{z\in\C\,:\, |z-\frac{p'}{2}|<\frac{p'}{2}\}\subseteq \sigma_{pt}(C_1';\ell(p'+))$.
		\item[\rm (ii)] $\{0\}\cup\{z\in\C\,:\, |z-\frac{p'}{2}|<\frac{p'}{2}\}\subseteq \sigma(C_1;\ell(p-))\subseteq \{z\in\C\,:\, |z-\frac{p'}{2}|\leq \frac{p'}{2}\}$.
		\item[\rm (iii)] $\sigma^*(C_1;\ell(p-))=\{z\in\C\,:\, |z-\frac{p'}{2}|\leq \frac{p'}{2}\}=\ov{\sigma(C_1; \ell(p-))}$.
	\end{itemize}
	\end{prop}

\begin{proof}
	(i) The first part of  (i) follows from Lemma \ref{Sp-LB}(i), the definition $\ell(p-)=\cup_{k=1}^\infty\ell^{p_k}$ with $1<p_k\uparrow p$, and the fact that $\sigma_{pt}(C_1;\ell^q)=\emptyset$ for every $1<q<\infty$; see Proposition \ref{Sp-C1-lp}(ii). 
	
	To establish the second part,
		fix $z\in\C$ with $|z-\frac{p'}{2}|<\frac{p'}{2}$. Since $1<p_k\uparrow p$, it follows that $p'_k\downarrow p'$ and hence, the open disk $B(\frac{p'}{2},\frac{p'}{2})\subseteq B(\frac{p'_k}{2},\frac{p'_k}{2})$ for every $k\in\N$. Accordingly, $|z-\frac{p'_k}{2}|<\frac{p_k'}{2}$
for all $k\in\N$. So, by \cite[Theorem 1(b)]{Le},  for each $k\geq 1$ there exists $x_k\in \ell^{p'_k}\setminus\{0\}$ such that $C'_1x_k=z x_k$ with $x_k=(x_{k,i})_{i\in\N_0}$ satisfying $x_{k,i+1}=x_{k,0}\prod_{h=0}^i(1-\frac{1}{z (h+1)})$ for all $i\in\N_0$ (see (1) on p.125 of \cite{Le}) for some  $x_{k,0}\in \C\setminus\{0\}$. Setting $x_{k,0}:=1$ for each $k\in\N$, it follows that $x_k=x_{1}=:x$ for all $k\in\N$ and hence,  $x\in \cap_{k\in\N}\ell^{p'_k}=(\ell(p-))'=\ell(p'+)$. On the other hand, it is clear that  $C_1'x=z x$. This shows the second part of (i).

(ii) To establish the second containment in (ii) we note that an analogous proof as that given for Proposition 3.2 in \cite{ABR4} also applies here. The use of Theorem 3.1 and Lemma 3.1(ii) there needs to be replaced, respectively, with the fact that $\sigma(C_1;\ell^q)=\{z\in\C\,:\, |z-\frac{q'}{2}|\leq \frac{q'}{2}\}$ for $1<q<\infty$ (cf. Proposition \ref{Sp-C1-lp}(ii)) and Lemma \ref{Sp-LB}(iii).

Concerning the first containment in (ii), observe that $C_1$ is \textit{not surjective} on $\ell(p-)$. Indeed, the element $y:=(\frac{1-(-1)^{n+1}}{2(n+1)})_{n\in\N_0}$ belongs to $\ell^{p_1}$ with $\ell^{p_1}\subseteq \ell(p-)$ and so $y\in \ell(p-)$. On the other hand, $x:=C_1^{-1}y=((-1)^n)_{n\in\N_0}$ belongs to $\omega$ but, $x\not\in \ell^{p_k}$ for every $k\in\N$ implies that $x\not\in \ell(p-)=\cup_{k=1}^\infty \ell^{p_k}$. Since $X$ is the unique element in $\omega$ satisfying $y=C_1x$ (as $C_1\in \cL(\omega)$ is a bicontinuous isomorphism), it follows that $y$ is \textit{not} in the range of $C_1\in \cL(\ell(p-))$ for every $1<p\leq \infty$. In particular, $0\in\sigma(C_1;\ell(p-))$. 

Fix $\lambda \in \C\setminus\{0\}$. If $\lambda\in \rho(C_1;\ell(p-))$, then $(\lambda I-C_1)(\ell(p-))=\ell(p-)$. Since $\ell(p-)$ is dense in $\ell^p$, it follows (with the bar denoting the closure in $\ell^p$) that
\[
\ell^p=\ov{\ell^p}=\ov{(\lambda I -C_1)(\ell(p-))}\subseteq \ov{(\lambda I- C_1)(\ell^p)}\subseteq \ell^p.
\]
By Proposition \ref{Sp-C1-lp} we can conclude that $|\lambda-\frac{p'}{2}|\geq \frac{p'}{2}$. Accordingly, $|\lambda-\frac{p'}{2}|<\frac{p'}{2}$ implies that $\lambda\in \sigma(C_1;\ell(p-))$.

(iii) An analogous argument used for the proof of Propostion 3.3 in  \cite{ABR4} also applies here. One only needs to replace the use of Proposition 3.2 and Theorem 3.1 there by part (ii) above and Proposition \ref{Sp-C1-lp}, respectively.
\end{proof}

\section{Dynamics of  the generalized Ces\`aro operators $C_t$ }
The aim of this section is to investigate the mean ergodicity and linear dynamics of the operator $C_t$, for $t\in [0,1]$, in $\omega$, in the Fr\'echet spaces $\{\ell(p+), ces(p+), d(p+):\, 1\leq p<\infty\}$ and in the (LB)-spaces $\{\ell(p-), ces(p-), d(p-):\, 1<p\leq \infty\}$. For the Banach spaces $\ell^1$, $d_1$ and $\ell^p$, $ces(p)$, $d_p$, for $1<p<\infty$, these results are also new. We also study the compactness, spectra and linear dynamics of the dual operators $C_t'$.

An operator $T\in \cL(X)$, with $X$ a lcHs, is called \textit{power bounded} if $\{T^n:\, n\in\N\}$
is an equicontinuous subset of $\cL(X)$. Here $T^n:= T \circ ...\circ T$ is the composition of T with itself $n$ times. For a Banach space $X$, this means precisely that
$\sup_{n\in\N}\|T^n\|_{X\to X}<\infty$.  Given $T\in\cL(X)$,  its sequence of averages
\begin{equation}\label{average}
	T_{[n]} := \frac{1}{n}\sum_{m=1}^nT^m,\quad  n\in \N,
\end{equation}
is called the \textit{Ces\`aro means} of $T$. The operator $T$ is said to be \textit{mean ergodic} (resp., \textit{uniformly mean ergodic}) if $(T_{[n]})_{n\in\N}$
	is a convergent sequence in $\cL_s(X)$ (resp., in $\cL_b(X)$). It follows from
	\eqref{average} that
	\[
	\frac{T^n}{n}= T_{[n]}- \frac{n-1}{n}T_{[n-1]},
	\]
	for $n\geq 2$. Hence, necessarily  $\frac{T^n}{n}\to 0$ in $\cL_s(X)$ (resp., in $\cL_b(X)$) as $n\to\infty$,
	whenever $T$ is
	mean ergodic (resp., uniformly mean ergodic). A relevant text is \cite{K}.
	
	Concerning the dynamics of a continuous linear operator $T$ defined on a separable	lcHs X, recall that $T$ is said to be \textit{hypercyclic} if there exists $x\in X$ whose orbit $\{T^nx:\, n\in\N_0\}$ is	dense in $X$. If, for some $x\in X$, the projective orbit $\{\lambda T^nx:\, \lambda\in\C,\ n\in\N_0\}$ is dense	in $X$, then $T$ is called \textit{supercyclic}. Clearly, any hypercyclic operator is also supercyclic.  As general	references, we refer to \cite{B-M,G-P}.
	
	We begin with a study of the dynamics of generalized Cesàro operators acting in $\omega$. For this, we will require, for each fixed $n\in\N_0$,  the combinatorial identity
	\begin{equation}\label{eq.Comb}
		\sum_{k=n-i}^n(-1)^{(n-i)-k}\binom{n+1}{k+1}=\binom{n}{i},\quad i=0,\ldots, n.
	\end{equation}
	For the proof we proceed by induction on $i=0,\ldots,n$. For $i=0$ observe that
	\[
	\sum_{k=n}^n (-1)^{n-k}\binom{n+1}{k+1}=(-1)^0\binom{n+1}{n+1}=1=\binom{n}{0}.
	\]
	Assume that \eqref{eq.Comb} is valid for some $0\leq i<n$. For $i+1$ it follows that
	\begin{align*}
		\sum_{k=n-(i+1)}^n &(-1)^{(n-i-1)-k}\binom{n+1}{k+1}=(-1)^0\binom{n+1}{n-i}+(-1)^{-1}\sum_{k=n-i}^n (-1)^{(n-i)-k}\binom{n+1}{k+1}\\
	&=\binom{n+1}{n-i}-\binom{n}{i}=\frac{(n+1)!}{(n-i)!(i+1)!}-\frac{n!}{i!(n-i)!}\\
	&=\frac{n!}{i!(n-i)!}\left[\frac{n+1}{i+1}-1\right]=\frac{n!}{(i+1)!(n-i-1)!}=\binom{n}{i+1}.
		\end{align*}
	Since this is identity \eqref{eq.Comb} for $i+1$, the proof is complete.

	\begin{theorem}\label{Dyn-omega} Let $t\in [0,1)$ and  $x^{[0]}:=\alpha_0(t^n)_{n\in\N_0}$ with $\alpha_0\in\C\setminus\{0\}$; see \eqref{eigenvalue}. 
		\begin{itemize}
		\item[\rm (i)] The  generalized Cesàro operator $C_t\in\cL(\omega)$ is power bounded and uniformly mean ergodic. 
		\item[\rm (ii)]
		$\Ker (I-C_t)={\rm span}\,\{x^{[0]}\}$ and the range
		\begin{equation}\label{eq.delta} (I-C_t)(\omega)=\{x\in\omega:\, x_0=0\}=\ov{{\rm span}\,\{e_r:\, r\in\N\}}
			\end{equation}
	of $(I-C_t)$	 is closed in $\omega$.
	\item[\rm (iii)]	 The operator $C_t$ is not supercyclic in $\omega$.
	\end{itemize}
		\end{theorem}
	
	\begin{proof} (i) That $C_t$ is power bounded follows from the barrelledness of $\omega$ and
		$r_n(C_tx)\leq r_n(x)$, for $x\in\omega$ and $n\in\N_0$ (cf. \eqref{eq.ContS-omega}),
		which implies, for every $x\in\omega$, that
\[
	r_n(C^m_tx)\leq r_n(x), \quad  m, n\in\N_0.
\]		
Since $\omega$ is Montel, $C_t$ is uniformly mean ergodic, \cite[Proposition 2.8]{ABR0}. 

(ii) By part (i) and \cite[Theorem 3.5]{ABR00} we can conclude that $(I-C_t)(\omega)$ is closed in $\omega$ and that
\begin{equation}\label{eq.delta2} 
	\omega=\Ker (I-C_t)\oplus (I-C_t)(\omega).
	\end{equation}
Moreover, Lemma \ref{L-3-3}(i) yields that  $\Ker(I-C_t)={\rm span}\{x^{[0]}\}$. Since $(C_tx)_0=x_0$ for each $x\in \omega$ (cf. \eqref{Ces-op}), we have $(I-C_t)(\omega)\su \{x\in\omega:\, x_0=0\}=\ov{{\rm span}\,\{e_r:\, r\in\N\}}$. In order to establish \eqref{eq.delta}, it remains to show that $e_r\in (I-C_t)(\omega)$ for each $r\geq 1$. Observe, via  Lemma \ref{L-3-1}(iii),  that
\begin{equation}\label{eq-id}
	(I-C_t)(e_n-te_{n+1})=(e_n-te_{n+1})-\frac{1}{n+1}e_n=\frac{n}{n+1}e_n-te_{n+1},\quad n\in\N_0.
	\end{equation}
Arguing by induction and using \eqref{eq-id} we can conclude that $e_r\in (I-C_t)(\omega)$ for each $r\geq 1$. Indeed, if $n=0$, then  \eqref{eq-id} yields $(I-C_t)(e_0-te_1)=-te_1$ and hence, $e_1\in (I-C_t)(\omega)$. Suppose that $e_n\in (I-C_t)(\omega)$. Then \eqref{eq-id} implies that $\frac{n}{n+1}e_n-te_{n+1}=	(I-C_t)(e_n-te_{n+1})\in (I-C_t)(\omega)$. Since $e_n\in 	(I-C_t)(\omega)$, by the induction hypothesis, it follows that  $e_{n+1}\in (I-C_t)(\omega)$. This completes the proof of \eqref{eq.delta}.

(iii) To verify that $C_t\in \cL(\omega)$
is not supercyclic we proceed as follows. It follows from \eqref{eq.delta2},   by a duality argument,  that $(\omega)'_\beta=
\Ker (I-C'_t)\oplus (I-C'_t)((\omega)'_\beta)$ and that ${\rm dim}\, \Ker (I-C'_t)={\rm codim}\, (I-C_t)(\omega)=1$, where $C'_t\in \cL((\omega)'_\beta)$ is the dual operator of $C_t$. Accordingly, $1\in \sigma_{pt}(C_t';(\omega)'_\beta)$. On the other hand, a direct calculation shows that the dual operator $C'_t\in\cL((\omega)'_\beta)$ is given by the transpose matrix of \eqref{Matrix}, that is,
\begin{equation}\label{DualC}
	C'_tz=\left(\sum_{k=i}^\infty \frac{t^{k-i}}{k+1}z_k\right)_{i\in\N_0},\quad z=(z_k)_{k\in\N_0}\in (\omega)'_\beta.
\end{equation}
Recall that $(\omega)'_\beta$ consists of vectors $z=(z_n)_{n\in\N_0}\in\C^{\N_0}$ with only finitely many  non-zero coordinates.
Define 
$$
z^{[n]}:=\sum_{i=0}^n(-1)^i\binom{n}{i}t^ie_{n-i}\in (\omega)'_\beta\setminus\{0\}, \quad n\in\N_0.
$$
It is shown below  that
\begin{equation}\label{eq.Auto-D}
 C'_tz^{[n]}=\frac{1}{n+1}z^{[n]}, \quad n\in\N_0.
 \end{equation}
This reveals that $\Lambda=\{\frac{1}{n+1}:\, n\in\N_0\}\su\sigma_{pt}(C_t'; (\omega)_\beta')$.  Since $\sigma(C_t; \omega)=\sigma_{pt}(C_t;\omega)=\Lambda$ (cf. Theorem \ref{Sp-omega}), it follows via \eqref{eq.DD-spettro} in  Corollary \ref{SpettroDuale} that also $\sigma_{pt}(C_t'; (\omega)'_\beta)\subseteq\sigma(C_t'; (\omega)'_\beta)=\Lambda$. So,  
$$
\sigma_{pt}(C_t'; (\omega)'_\beta)=\sigma(C_t'; (\omega)'_\beta)=\Lambda.
$$
 In particular, $C'_t$ has a plenty of eigenvalues which implies that $C_t$ cannot be supercyclic,  \cite[Proposition 1.26]{B-M}.
 
 It remains to establish \eqref{eq.Auto-D}. Note, for $n\in\N_0$ fixed, that $(z^{[n]})_i=0$ if $i>n$ and $(z^{[n]})_{n-i}=(-1)^i\binom{n}{i}t^i$ for $i=0,\ldots,n$. In particular, $z^{[n]}\in (\omega)'_\beta\setminus\{0\}$. For $i>n$ it is clear that 
 \[
 (C_t'z^{[n]})_i=\sum_{k=i}^\infty\frac{t^{k-i}}{k+1}(z^{[n]})_k=0=\frac{1}{n+1}\cdot 0=\frac{1}{n+1}(z^{[n]})_i.
 \] 
 To verify that $(C'_tz^{[n]})_{n-i}=\frac{1}{n+1}(z^{[n]})_{n-i}$ for $i=0,\ldots, n$ observe that 
 \begin{align*}
 	(C'_tz^{[n]})_{n-i}&=\sum_{k=n-i}^\infty \frac{t^{k-(n-i)}}{k+1}(z^{[n]})_{k}=\sum_{k=n-i}^n\frac{t^{k-(n-i)}}{k+1}(z^{[n]})_{k}\\
 	&=\sum_{k=n-i}^n\frac{t^{k-(n-i)}}{k+1}(z^{[n]})_{n-(n-k)}=\sum_{k=n-i}^n\frac{t^{k-(n-i)}}{k+1}(-1)^{n-k}\binom{n}{n-k}t^{n-k}\\
 	&=\sum_{k=n-i}^n\frac{t^{i}}{k+1}(-1)^{n-k}\frac{n!}{(n-k)!\,k!}\cdot\frac{n+1}{n+1}\\
 	&=\frac{t^i(-1)^i}{n+1}\sum_{k=n-i}^n(-1)^{(n-i)-k}\binom{n+1}{k+1}=\frac{(-1)^i}{n+1}t^i\binom{n}{i},
 \end{align*}
where the last equality follows from \eqref{eq.Comb}. But, as noted above, $(-1)^i\binom{n}{i}t^i=(z^{[n]})_{n-i}$ and so $(C_t'z^{[n]})_{n-i}=\frac{1}{n+1}(z^{[n]})_{n-i}$ for $i=0,\ldots, n$. The identity \eqref{eq.Auto-D} is  thereby established and the proof is complete.
	\end{proof}
	
We now turn to the dynamics  of generalized Cesàro operators $C_t$ acting in the other sequence spaces	considered in this paper, for which  we first need to establish some general results on bounded linear operators acting in  lcHs'.  Recall that a linear operator $T\colon X\to Y$, with $X, Y $ lcHs', is said to be \textit{bounded} if there exists a neighbourhood $\cU$ of $0\in X$ such that $T(\cU)$ is a bounded subset of $Y$. It is routine to verify that necessarily $T\in \cL(X,Y)$. A lcHs $X$ is called \textit{locally complete} if, for each closed,  absolutely convex subset $B\in \cB(X)$, the space $X_B:={\rm span}\, (B)$ equipped with the Minkowski functional $\|\cdot\|_B$, \cite[p.47]{MV}, is  a Banach space, whose closed unit ball is $B$. Such a set $B$ is also called a Banach disc, \cite[\S 8.3]{J}.

\begin{theorem}\label{BoundedOp} Let $X$ be a locally complete lcHs and  $T\in \cL(X)$ be a bounded operator satisfying  $\sigma(T;X)\su \ov{B(0,\delta)}$ for some $\delta\in (0,1)$. Then $T^n\to 0$ in $\cL_b(X)$ as $n\to\infty$. In particular, $T$ is both power bounded and uniformly mean ergodic.
\end{theorem}

\begin{proof}
	Since $T$ is a bounded operator, there exists a closed, absolutely convex  neighbourhood $\cU$ of $0\in X$ such that $T(\cU)\in \cB(X)$. So, we can select a closed,  absolutely convex subset $B\in \cB(X)$ such that $T(\cU)\su B$. By the assumptions, $(X_B, \|\cdot\|_B)$ is a Banach space. Since $T(\cU)\su B$, the map $S\colon X\to X_B$ defined by $Sx:=Tx$ for $x\in X$, is well defined and it is clearly continuous. Let  $j\colon X_B\to X$ denote the canonical inclusion of $X_B$ into $X$, i.e., $j(x):=x$ for $x\in X_B$. Then $j\in \cL(X_B, X)$ and $T=jS\in \cL(X)$. On the other hand  $Sj\in \cL(X_B)$. So, by \cite[Proposition 5, p.199]{Gr} we have that
	\[
	\sigma(jS;X)\setminus\{0\}=\sigma(Sj;X_B)\setminus\{0\}.
	\]
	Accordingly, $\sigma(Sj;X_B)=\sigma(T;X)\su \ov{B(0,\delta)}$. This implies that the spectral radius $r(Sj)$ of $Sj$ satisfies $r(Sj)\leq \delta<1$. Since  $r(Sj)=\lim_{n\to\infty}\left(\|(Sj)^n\|_{X_B\to X_B}\right)^{1/n}$, it follows via standard arguments that $(Sj)^n\to 0$ in $\cL_b(X_B)$ as $n\to\infty$. The claim is that this implies $T^n\to 0$ in $\cL_b(X)$ as $n\to\infty$. To establish the claim,  fix any $C\in \cB(X)$ and any absolutely convex neighbourhood $\cV$ of $0\in X$. Then there exist $\lambda >0$ such that $C\su \lambda \cU$ and $\mu>0$ such that $B\su \mu \cV$. Since $B$ is the unit closed ball of $X_B$ and $(Sj)^n\to 0$ in $\cL_b(X_B)$, there exists $n_0\in\N$ such that $(Sj)^n(B)\su \frac{1}{\lambda\mu}B$ for all $n\geq n_0$. So, for each $n> n_0$, it follows that
	\begin{align*}
	T^n(C)&\su \lambda T^n(\cU)= \lambda T^{n-1}T(\cU)\subseteq \lambda T^{n-1}(B)=\lambda T^{n-1}(j(B))=\lambda (jS)^{n-1}(j(B))\\
	& =\lambda j(Sj)^{n-2}S(j(B))=\lambda j[(Sj)^{n-1}(B)]\su \lambda j\left(\left(\frac{1}{\lambda\mu}\right)B\right)=\left(\frac{1}{\mu}\right)j(B)\\
	&=\left(\frac{1}{\mu}\right)B\su \cV.
	\end{align*}
This means, with  $W(C,\cV):=\{R\in \cL(X):\, R(C)\su \cV\}$, that $T^n\in W(C,\cV)$ for each $n>n_0$. Since $C\in \cB(X)$ and $\cV$ are arbitrary and the sets $W(C,\cV)$ form a basis of neighbourhoods for $0$ in $\cL_b(X)$, the claim is proved, i.e., $T^n\to 0$ in $\cL_b(X)$ as $n\to\infty$. It follows that $T$ is power bounded (clearly) and that $T_{[n]}\to 0$ in $\cL_b(X)$ as $n\to\infty$ (i.e., $T$ is uniformly mean ergodic). Indeed, let $q$ be any $\tau_b$-continuous seminorm. Then \eqref{average} implies that $q(T_{[n]})\leq \frac{1}{n}\sum_{m=1}^nq(T^m)$ for $n\in\N$. Since $q(T^n)\to 0$ in $[0,\infty)$, also its arithmetic means $\frac{1}{n}\sum_{m=1}^nq(T^m)\to 0$ for $n\to\infty$, that is, $\lim_{n\to\infty}q(T_{[n]})=0$. So, we can conclude that $T_{[n]}\to 0$ in $\cL_b(X)$ for $n\to\infty$.
\end{proof}

Theorem \ref{BoundedOp} permits us to formulate and prove the following general criterion for  power boundedness and uniform mean ergodicity. To state it,  recall that a lcHs $X$ is said to be \textit{ultrabornological} if it  is an inductive limit of Banach spaces, \cite[\S 13.1]{J}, \cite[p.283]{MV}. For instance, Fr\'echet spaces, \cite[Corollary 13.1.4]{J}, and  (LB)-spaces are ultrabornological. A lcHs $X$ is called \textit{a webbed space} if  a \textit{web} can be defined on $X$ . For the definition of a web  and  the properties of webbed spaces we refer to \cite[\S 5.2]{J} and \cite[Ch. 2.4]{24}. Recall from Section 2 that  Fr\'echet spaces and (LB)-spaces are webbed spaces.  Moreover, sequentially closed subspaces and quotients of webbed spaces are webbed spaces, \cite[Theorem 5.3.1]{J}.

For what follows we require the next result concerning algebraic sums in ultrabornological lcHs' which  can be found in \cite[\S 35.5(4), p.66]{24}.

\begin{prop}\label{Decomp} Let $X$ be an ultrabornological lcHs such  that $X=X_1\oplus X_2$ algebraically with both $X_1, X_2\su X$ webbed spaces for the topology induced by $X$. Then $X_1$ and $X_2$ are closed subspaces of $X$ and $X=X_1\oplus X_2$ topologically, i.e., the canonical projections $P_i\colon X\to X_i$ are continuous for $i=1,2$.
\end{prop}

In general compact operators need not be mean ergodic. Just consider $T=\alpha I$ with $|\alpha|>1$ in a finite dimensional space.

\begin{theorem}\label{OP} Let $X$ be a locally complete, webbed and ultrabornological lcHs. Let $T\in \cL(X)$ be a compact operator such that $1\in\sigma(T;X)$ with $\sigma(T;X)\setminus\{1\}\su \ov{B(0,\delta)}$ for some $\delta\in (0,1)$  and satisfying $\Ker (I-T)\cap (I-T)(X)=\{0\}$. Then $T$ is both power bounded and uniformly mean ergodic.
	\end{theorem}

\begin{proof} Since $T\in \cL(X)$ is a compact operator, the following properties hold true:
		(a) $(I-T)(X)$ is closed in $X$, (b) ${\rm dim}\Ker (I-T)<\infty$ (1 is necessarily an eigenvalue of $T$ as it is an isolated point of $\sigma(T;X)$ and $T$ is  compact), and (c) ${\rm codim}\, (I-T)(X)={\rm dim}\, \Ker (I-T)<\infty$, see, e.g., \cite[Theorem 9.10.1]{Ed}. Since  $\Ker (I-T)\cap (I-T)(X)=\{0\}$ by assumption, it follows that $X=\Ker (I-T)\oplus (I-T)(X)$ algebraically. Moreover, $(I-T)(X)$ and $\Ker (I-T)$ are closed complemented subspaces of $X$ and hence, are webbed spaces, \cite[Theorem 5.3.1]{J}. So, we can apply Proposition \ref{Decomp} to conclude that 	$X=\Ker (I-T)\oplus (I-T)(X)$ holds topologically.
		
		Set $Y:=(I-T)(X)$ and $S:=T|_Y$. It is routine to  verify that $S(Y)\su Y$ and $S\colon Y\to Y$ is a compact operator. So, $\sigma(S;Y)\setminus\{0\} =\sigma_{pt}(S;Y)\su \sigma_{pt}(T;X)\su \sigma(T;X)$. But, $1\not \in\sigma(S;Y)$.  Otherwise, there exists $y\in Y\setminus\{0\}$ such that $Sy=y$, i.e., $Ty=y$ or, equivalently, $(I-T)y=0$. Thus, $y\in Y\cap \Ker (I-T)=(I-T)(X)\cap \Ker (I-T)=\{0\}$ and hence, $y=0$; a contradiction. Hence, $\sigma(S;Y) \su \sigma(T;X)\setminus\{1\}\su \ov{B(0,\delta)}$  with $\delta\in (0,1)$. Since $S$ is compact, it is also bounded and hence, we can  apply Theorem \ref{BoundedOp} to conclude that $S^n\to 0$ in $\cL_b(Y)$ as $n\to\infty$, after noting that the closed subspace $Y$ of $X$ is locally complete.
		
		Denote by $P\colon X\to X$ the continuous projection onto $\Ker (I-T)$ along $(I-T)(X)=Y$, i.e., for each $z\in X$ there exist unique elements $x\in \Ker (I-T)$ and $y\in Y$ such that $z=x+y$ and so $Pz:=x$. The claim is that $T^n\to P$ in $\cL_b(X)$ as $n\to\infty$. To establish this  fix $B\in \cB(X)$ and a neighbourhood $\cU$ of $0\in X$. As $(I-P)\in \cL(X)$, we have  that $(I-P)(B)\in \cB(Y)$.  Taking into account that $S^n\to 0$ in $\cL_b(Y)$ as $n\to\infty$, there exists $n_0\in\N$ such that $S^n((I-P)(B))\su \cU\cap Y$ for every $n\geq n_0$. On the other hand, for each $z\in X$ we have that $Pz\in \Ker (I-T)$, i.e., $TPz=Pz$, and hence, $T^n(Pz)=Pz$ for each $n\in\N$. Accordingly, as $S=T$ on $(I-P)(X)=(I-T)(X)=Y$ we get, for each $z\in B$ and $n\geq n_0$, that
		\begin{align*}
		T^nz-Pz&=T^n(Pz+(z-Pz))-Pz=T^n(z-Pz)=T^n((I-P)z)\\
		&=S^n((I-P)z)\in S^n((I-P)(B))\su \cU\cap Y,
		\end{align*}
	where we used the fact that $(I-P)z\in Y$.
Since $z\in B$ is arbitrary, this implies that $T^n-P\in W(B,\cU):=\{R\in \cL(X):\, R(B)\su \cU\}$ for each $n\geq n_0$. So, by the arbitrariness of $B$ and $\cU$, the claim is proved.
\end{proof}

\begin{remark}
	\rm
	(i) Let $X$ be a sequentially complete lcHs and $T\in \cL(X)$. If $\frac{T^n}{n}\to 0$ in $\cL_s(X)$ as $n\to\infty$, then $\sigma(T;X)\su \ov{B(0,1)}$,  \cite[Proposition 5.1 \& Remark 5.3]{A-M}; see also \cite[Proposition 4.4]{FGJ}. In particular, if $T$ is power bounded, then  $\sigma(T;X)\su \ov{B(0,1)}$. In view of this fact, Theorem \ref{OP} can be seen as a sort of converse result (observe that every sequentially complete lcHs is  locally complete, \cite[Corollary 23.14]{MV}).
	
	(ii) Theorem \ref{BoundedOp} should also be compared with \cite[Theorem 10]{ABR-00} in which it is proved,  for $T\in \cL(X)$ with $X$ a prequojection Fr\'echet space, that $T^n\to 0$ in $\cL_b(X)$ as $n\to\infty$ if, and only if, $\sigma(T;X)\su B(0,1)$ and $\frac{T^n}{n}\to 0$ in $\cL_b(X)$. Since $\sigma(C_t;\omega)\not\subseteq B(0,1)$ (as $1\in \sigma(C_t;\omega)$ but  $1\not\in {B(0,1)}$) and $\omega$ is a prequojection Fr\'echet space, for each $t\in [0,1)$, it follows that  $(C_t)^n\not\to 0$ in $\cL_b(\omega)$ for $n\to\infty$.
\end{remark}

Combining Theorem \ref{OP} with the results in the preceding sections we get the following result.

\begin{theorem}\label{Ct-PM} Let $t\in [0,1)$. Let $X$  belong to any one of the sets: $\{d_p,\ell^p:\, 1\leq p<\infty\}\cup\{ces(p):\, 1<p<\infty\}$ or $\{\ell(p+), ces(p+), d(p+):\, 1\leq p<\infty\}$ or $\{\ell(p-), ces(p-), d(p-):\, 1< p\leq \infty\}$. Then $C_t\in \cL(X)$ is power bounded and uniformly mean ergodic, but not supercyclic.
	\end{theorem}

\begin{proof} From the results of the preceding sections  recall that $C_t\in \cL(X)$ is a compact operator on $X$ and $\sigma(C_t;X)=\Lambda\cup\{0\}$.
	Hence, $\sigma(C_t;X)\setminus\{1\}\su \ov{B(0,1/2)}$. Moreover, $(I-C_t)(X)$ is also  closed in $X$. Since $x^{[0]}\in d_1\subseteq X$, we can adapt 
	 the arguments in the proof of Theorem \ref{Dyn-omega} to argue that $(I-C_t)(X)=\{x\in X:\, x_0=0\}=\ov{{\rm span}\{e_r:\, r\in\N\}}$ and $\Ker (I-C_t)={\rm span}\, \{x^{[0]}\}$. Hence, $\Ker (I-C_t)\cap (I-C_t)(X)=\{0\}$. So, all the assumptions of Theorem \ref{OP} (for $\delta=\frac{1}{2}$ and $T:=C_t$) are satisfied. Then we can conclude that  $C_t$ is power bounded and uniformly mean ergodic.
	
To show that $C_t\colon X\to X$ is not supercyclic we proceed as follows. Since $C_t\in \cL(X)$ is compact, the operators $C_t\colon X\to X$  and $C'_t\colon X'_\beta\to X'_\beta$ have the same non-zero eigenvalues, \cite[Theorem 9.10.2(2)]{Ed}. Hence, $\sigma_{pt}(C'_t;X'_\beta)=\sigma_{pt}(C_t;X)=\Lambda$. According to \cite[Proposition 1.26]{B-M} it follows that the operator $C_t\colon X\to X$ cannot be supercyclic.
\end{proof}


A first consequence of the results collected above is the following one concerning the dual operators $C'_t$. First we recall the relevant dual spaces involved. Namely, for $p, p'$ satisfying $\frac{1}{p}+\frac{1}{p'}=1$ we have (see Proposition 3.4(i), Proposition 4.3 and Remark 4.4 in \cite{BR2}, respectively): 

$\ell(p-)\simeq (\ell(p'+))'_\beta$ and $(\ell(p-))'_\beta\simeq \ell(p'+)$, for $1<p\leq\infty$;

$d(p-)\simeq (ces(p'+))'_\beta$ and $(ces(p-))'_\beta\simeq d(p'+)$, for $1< p\leq\infty$; 

$ces(p-)\simeq (d(p'+))'_\beta$ and $ces(p'+)\simeq (d(p-))'_\beta$, for $1<p\leq \infty$.

\begin{prop}\label{PP-dualOperator} Let $t\in [0,1)$ and $X$ belong to any one of the sets: $\{d_p,\ell^p\,:\, 1\leq p<\infty\}\cup\{ces(p)\,:\, 1<p<\infty\}$ or $\{\ell(p+), ces(p+), d(p+)\,:\, 1\leq p<\infty\}$ or $\{\ell(p-), ces(p-), d(p-)\,:\, 1<p\leq\infty\}$.  
	\begin{itemize}
		\item[\rm (i)] The dual operator $C'_t\in \cL(X'_\beta)$ of $C_t\in \cL(X)$ is compact and is given by 
		\begin{equation}\label{eq.DualeO-X}
		C'_ty=\left(\sum_{k=i}^\infty\frac{t^{k-i}}{k+1}y_k\right)_{i\in\N_0},\quad y=(y_k)_{k\in\N_0}\in X'_\beta.
		\end{equation}
		\item[\rm (ii)] The point spectrum of $C'_t\in \cL(X'_\beta)$ is given by
		\begin{equation}\label{eq.XX-s}
		\sigma_{pt}(C'_t;X'_\beta)=\sigma_{pt}(C_t;X)=\Lambda.
		\end{equation}
		\end{itemize}
		Each eigenvalue $\frac{1}{n+1}$, for $n\in\N_0$, is simple and its corresponding eigenspace is spanned by
		\[
		 y^{[n]}=\sum_{i=0}^n(-1)^i \binom{n}{i}t^ie_{n-i}\in X'_\beta\setminus\{0\}, \ n\in\N_0.
		\]
		Moreover,
		\[
		\sigma^*(C_t';X'_\beta)=\sigma(C'_t; X'_\beta)=\Lambda\cup\{0\}.
		\]
	\end{prop}

\begin{proof} (i) Recall that $\cE$ is an unconditional basis  in $\ell(p+)$ $ces(p+)$, $d(p+)$, for $1\leq p<\infty$ (cf. Section 4) and an unconditional basis in $\ell(p-)$, $ces(p-)$, $d(p-)$, for $1<p\leq\infty$ (cf. Section 5). Moreover, $\cE$ is also an unconditional basis  in the dual Banach spaces $(\ell^p)'=\ell^{p'}$ for $1<p<\infty$, in the dual Banach spaces $(ces(p))'\simeq d_{p'}$ for $1<p<\infty$, \cite{BR1}, and in the dual Banach spaces $(d_p)'\simeq ces(p')$ for $1<p<\infty$ (cf. \cite{Be}, \cite{CR3}), as well as in $(d_1)'\simeq ces(0)$,  \cite[Section 6]{CR2}. In view of the description of $X'_\beta$ (for $X$ non-normable) given prior to this Proposition it follows, for all $X\not=\ell^1$, that the linear space $\mbox{span}(\cE)=(\omega)'$ is dense in $X'_\beta$. The continuity of $C'_t\colon X'_\beta\to X'_\beta$ then implies that \eqref{DualC} can be extended to an inequality for every $y\in X'_\beta$, that is, \eqref{eq.DualeO-X} is valid. 
	
	For $X=\ell^1$, the linear space $\mbox{span}(\cE)=(\omega)'$ is not dense in $X'_\beta=\ell^\infty$. So, in this case we argue as follows.
		Define $Ty:=(\sum_{k=i}^\infty\frac{t^{k-i}}{k+1}y_k)_{i\in\N_0}$ for $y\in\ell^\infty$, in which case $T\in\cL(\ell^\infty)$. Indeed, for $y\in\ell^\infty$, note that 
		\begin{align*}
		\|Ty\|_\infty=&\sup_{i\in\N_0}	\left|\sum_{k=i}^\infty\frac{t^{k-i}}{k+1}y_k\right|\leq \sup_{i\in\N_0}\sum_{k=i}^\infty\frac{t^{k-i}}{k+1}|y_k|\leq \|y\|_\infty \sup_{i\in\N_0}\sum_{k=i}^\infty\frac{t^{k-i}}{k+1}\\
	\leq 	& 		\|y\|_\infty \sup_{i\in\N_0} \sum_{k=i}^\infty t^{k-i}=\|y\|_\infty\sum_{j=0}^\infty t^j=\frac{1}{1-t}\|y\|_\infty  \ \ (\mbox{as } 0\leq t<1).
\end{align*}
Accordingly, $\|T\|_{\ell^\infty\to \ell^\infty}\leq \frac{1}{1-t}$, that is, $T\in \cL(\ell^\infty)$. For each $x\in \ell^1$ and $y\in\ell^\infty$, a direct calculation yields
\[
\langle C_t x, y\rangle=\langle x, Ty\rangle, 
\]
which implies that $T=C'_t$.

For any Fr\'echet space $X\in\{\ell(p+), ces(p+), d(p+)\,:\, 1\leq p<\infty\}$ and any Banach space $X\in \{\ell^1, d_1\}\cup \{\ell^p, ces(p), d_p\, :\, 1<p<\infty\}$ the operator $C_t\in \cL(X)$ is compact (cf. Propositions \ref{Sp-spazilp}, \ref{Sp-spazices}, \ref{Sp-spazidp} and Remark \ref{Nuova R} and Theorem \ref{Sp-F}(i)). Accordingly,  the dual operator $C'_t\in \cL(X'_\beta)$ of $C_t\in \cL(X)$ is compact, \cite[Corollary 9.6.3]{Ed}. 

For any (LB)-space $X\in \{\ell(p-), ces(p-), d(p-)\,:\, 1<p\leq\infty\}$ the operator $C_t\in \cL(X)$ is also compact (cf. Theorem \ref{Spectrum-LB}(i)). So, the compactness of  $C_t'\in\cL(X'_\beta)$ follows from Proposition \ref{Comp-dual}, after observing that $X$ is a boundedly retractive (LB)-space. Indeed,   $X=\ell(p-)$, for $1<p\leq\infty$, is  a boundedly retractive (LB)-space, as it is the strong dual of  the quasinormable Fr\'echet space $\ell(p'+)$, \cite[p.12]{M-M}. On the other hand, $X\in \{ces(p-), d(p-)\,:\, 1<p\leq \infty\}$ is  a boundedly retractive (LB)-space, as it is a (DFS)-space, \cite[Proposition 2.5(ii) \& Lemma 4.2(i)]{BR2}.	

(ii) It was shown in the proof of Theorem  \ref{Dyn-omega} that each vector $z^{[n]}\in (\omega)'_\beta \setminus\{0\}\subseteq X'_\beta$ satisfies $C'_tz^{[n]}=\frac{1}{n+1}z^{[n]}$, for every $n\in\N_0$. Accordingly.
\begin{equation}\label{alfa}
	\Lambda \subseteq \sigma_{pt}(C'_t;X'_\beta).
	\end{equation} 
Moreover, $0\not\in\ \sigma_{pt}(C'_t;X'_\beta)$ as $C'_t$ is injective. To verify this let $z\in X'_\beta$ satisfy $C'_tz=0$. By considering the individual coordinates in \eqref{eq.DualeO-X} it follows that 
\[
\frac{1}{i+1}z_i=(C_t'z)_i-t(C'_tz)_{i+1}, \quad i\in\N_0,
\]
that is, $z=0$ and so indeed $0\not\in\ \sigma_{pt}(C'_t;X'_\beta)$. The compactness of $C'_t\in\cL(X'_\beta)$ then implies that 
\begin{equation}\label{beta}
	\sigma(C'_t; X'_\beta)=\{0\}\cup \sigma_{pt}(C'_t;X'_\beta) \ \mbox{ and }\ 0\not\in\ \sigma_{pt}(C'_t;X'_\beta).
\end{equation} 
It follows from \eqref{eq.DD-spettro} in Corollary \ref{SpettroDuale} (with $T:=C_t$), from \eqref{beta} and from the fact that $\sigma_{pt}(C_t;X)=\Lambda$, that \eqref{eq.XX-s} is valid.

Parts (1) and (2) of \cite[Proposition 9.10.2]{Ed} imply that each eigenvalue of $C'_t$ is simple, as this is the case for $C_t$; see Propositions \ref{Sp-spazilp}, \ref{Sp-spazices}, \ref{Sp-spazidp} and Remark \ref{Nuova R} and Theorems \ref{Sp-F}, \ref{Spectrum-LB}, which also include the identities
\begin{equation}\label{gamma}
\sigma^*(C_t;X)=\sigma(C_t;X)=\Lambda\cup\{0\}.
\end{equation}
Setting $T:=C_t$ it follows from  \eqref{eq.DD-Spettro} in Corollary \ref{SpettroDuale}, together with \eqref{gamma}, that
\[
\sigma^*(C'_t;X'_\beta)\subseteq \sigma^*(C_t;X)=\Lambda\cup\{0\}.
\]
From general theory (cf. Section 2) we also have that
\[
\sigma(C_t';X'_\beta)\subseteq \sigma^*(C'_t;X'_\beta).
\]
	Since \eqref{eq.XX-s} and \eqref{beta} imply that $\sigma(C_t';X'_\beta)=\Lambda\cup\{0\}$, we can conclude that 
	\[
	\Lambda\cup\{0\}=\sigma(C_t';X'_\beta)\subseteq \sigma^*(C_t';X'_\beta)\subseteq \Lambda\cup\{0\}.
	\]
	This, together with \eqref{gamma}, yields $\sigma^*(C_t';X'_\beta)=\sigma^*(C_t;X)= \Lambda\cup\{0\}$.
\end{proof}

A consequence of Theorem \ref{Ct-PM} is the next result.

\begin{prop}\label{P-6-8} Let $t\in [0,1)$. Let $X$ belong to any one of the sets: $\{d_p, \ell^p\,:\, 1\leq p<\infty\}\cup \{ces(p)\,:\, 1<p<\infty\}$ or $\{\ell(p+), ces(p+), d(p+)\,:\, 1\leq p<\infty\}$ or $\{\ell(p-), ces(p-), d(p-)\,:\, 1<p\leq\infty\}$. Then $C_t'\in \cL(X_\beta')$ is power bounded and uniformly mean ergodic, but not supercyclic.
	\end{prop}

\begin{proof} By Theorem \ref{Ct-PM} the operator $C_t\in \cL(X)$ is power bounded. Since  $(C_t')^n=(C_t^n)'$, for every $n\in\N_0$,  it follows from \cite[\S 39.3(6)]{24} that also $C'_t\in \cL(X'_\beta)$ is power bounded. The operator $C_t\in \cL(X)$ is also uniformly mean ergodic in $X$, again by Theorem \ref{Ct-PM}. Since $X$ is barrelled (hence, quasi-barrelled), Lemma 2.1 in \cite{ABR} implies that $C'_t$ is uniformly mean ergodic in $X'_\beta$. If $X\not\in\{\ell^1, d_1\}$, then $X'_\beta$ is reflexive with $(X'_\beta)'_\beta=X$ (cf. the proof of Proposition \ref{PP-dualOperator}) and hence, $(C'_t)'=C_t$. It follows  from \eqref{eq.XX-s}  that $C''_t=C_t$ has  plenty of eigenvalues so that $C_t'\in \cL(X'_\beta)$ cannot be supercyclic \cite[Proposition 1.26]{B-M}. Finally, suppose that $X\in\{\ell^1, d_1\}$. Since $C_t$ is compact with $\sigma_{pt}(C_t; X)=\Lambda$ (cf. Proposition \ref{Sp-spazilp} and Remark \ref{Nuova R}), it follows that $\sigma_{pt}(C'_t;X'_\beta)=\sigma_{pt}(C_t; X)=\Lambda$; see \cite[Proposition 9.10.2(2)]{Ed}. Schauder's theorem implies that $C'_t\in \cL(X'_\beta)$ is also compact and hence, again by Proposition 9.10.2(2) of \cite{Ed}, now applied to $C'_t$, we can conclude that $\sigma_{pt}(C''_t;X''_\beta)=\sigma_{pt}(C'_t;X'_\beta)=\Lambda$. So, $C''_t\in \cL(X''_\beta)$ has  plenty of eigenvalues which implies that $C'_t$ is not supercyclic. 
	\end{proof}

\begin{remark} \rm
	The dynamics of  $C_1\in \cL(X)$, with $X\not\in \{\ell^1,d_1\}$  belonging to one of the sets in Theorem \ref{Ct-PM}, is quite different. Consider first the Banach space case. For $1<p<\infty$, the operator   $C_1\in \cL(\ell^p)$ is neither power bounded nor mean ergodic, \cite[Proposition 4.2]{ABR00}. Since $\{z\in \C\,:\, |z-\frac{p'}{2}|<\frac{p'}{2}\}\subseteq \sigma_{pt}(C_1'; \ell^{p'})$ with $\frac{1}{p}+\frac{1}{p'}=1$,  \cite[Theorem 1(b)]{Le}, $C_1\in \cL(\ell^p)$ cannot be supercyclic, \cite[Proposition 1.26]{B-M}. Similarly,  $C_1\in \cL(ces(p))$, for  $1<p<\infty$,  is not mean ergodic, not power bounded and not supercyclic, \cite[Proposition 3.7(ii)]{ABR6}. Also, $C_1\in \cL(d_p)$ is not mean ergodic and not supercyclic, \cite[Propositions 3.10 \& 3.11]{BR1}. Since power bounded operators in reflexive Banach spaces are necessarily mean ergodic, \cite{Lor}, $C_1$ cannot  be power bounded in $d_p$. Turning to Fr\'echet spaces, for $1\leq p<\infty$ the operator $C_1\in \cL(\ell(p+))$ is not mean
	ergodic, not power bounded and not supercyclic,  \cite[Theorems 2.3 \& 2.5]{ABR3}, as is the case for $C_1\in \cL(ces(p+))$, \cite[Proposition 5]{ABR8}, and for  $C_1\in \cL(d(p+))$,  \cite[Proposition 3.5]{BR3}. For (LB)-spaces, with $1< p\leq\infty$,  the operator $C_1\in \cL(ces(p-))$ is not mean
	ergodic, not power bounded and not supercyclic,  \cite[Propositions 3.4 \& 3.5]{ABR4}, as is the case for $C_1\in \cL(d(p-))$,  \cite[Proposition 3.8]{BR3}.
Finally, the dynamics of $C_1\in \cL(\omega)$ is the same as for $C_t\in \cL(\omega)$, with $t\in [0,1)$; see Theorem \ref{Dyn-omega} above and \cite[Proposition 4.3]{ABR3}.
\end{remark}

The dynamics of $C_1$ acting in $\ell(p-)$ is covered by our final result.

\begin{prop} Let $p\in (1,\infty]$. The Cesàro operator $C_1\in \cL(\ell(p-))$ is  not mean
	ergodic, not power bounded and not supercyclic.
	\end{prop}

\begin{proof} In view of Proposition \ref{Sp-lp-}(i) the proof follows in a similar way to that  of \cite[Theorem 2.3]{ABR3}. For the sake of completeness, we indicate the details.
	
	By the discussion prior to Proposition \ref{PP-dualOperator} we know that $(\ell(p-))'_\beta\simeq \ell(p'+)$. Proposition \ref{Sp-lp-}(i) implies that    $\frac{1+p'}{2}>1$ belongs to $\sigma_{pt}(C_1';\ell(p'+))$, where $\frac{1}{p}+\frac{1}{p'}=1$.  So, 
	there exists a non-zero vector $u\in \ell(p'+)$ satisfying $C_1'(u) =\frac{1+p'}{2}u$. Choose any $x\in \ell(p-)$
	such that $\langle x,u\rangle\not=0$. Then
\[
\langle  \frac{1}{n}(C_1)^n(x),u\rangle=\langle  x, \frac{1}{n}(C'_1)^n(u)\rangle=\frac{1}{n}\left(\frac{1+p'}{2}\right)^n\langle x,u\rangle,\quad n\in\N.
\]
	This means that the sequence $\{\frac{1}{n}(C_1)^n(x)\}_{n\in\N}\su \ell(p-)$ cannot be bounded in  $\ell(p-)$. Accordingly, $C_1$ is not mean ergodic and not power bounded.
	
Applying again Proposition \ref{Sp-lp-}(i), we see that  $C'_1$ has a plenty of eigenvalues. So, $C_1$ cannot be supercyclic,  \cite[Proposition 1.26]{B-M}.
\end{proof}

\bigskip

\textbf{Acknowledgements.} The research of J. Bonet was partially supported by the project 
PID2020-119457GB-100 funded by MCIN/AEI/10.13039/501100011033 and by 
``ERFD A way of making Europe'' and by the project GV AICO/2021/170.

\bibliographystyle{plain}

\end{document}